\documentclass[a4paper,12pt]{amsart} 
\usepackage{etex,soul}
\usepackage{pgf,ulem}
\usepackage{amscd,amsmath,amssymb,amstext,amsfonts,amsthm,latexsym}
\usepackage{verbatim}
\usepackage[english]{babel}
\usepackage[utf8]{inputenc}
\usepackage{graphicx}
\usepackage{mathrsfs,ulem}

\usepackage[all]{xy}
\usepackage{supertabular,multirow}
\usepackage{fancyhdr}
\usepackage{url} 
\usepackage{hyperref}

\newcommand\sE{{\mathcal E}}
\newcommand\sA{{\mathcal A}}

\newcommand\sL{{\mathcal L}}

\newcommand\sX{{\mathcal X}}

\newcommand\sH{{\mathcal H}}

\newcommand\la{\lambda}

\newcommand\Ga{\Gamma}
\newcommand\De{\Delta}

\newcommand\de{\delta}

\DeclareMathOperator{\Pic}{Pic}

\DeclareMathOperator{\Ext}{Ext}
\DeclareMathOperator{\Hom}{Hom}
\DeclareMathOperator{\Alb}{Alb}
\DeclareMathOperator{\divi}{div}

\newcommand{\CC}{\ensuremath{\mathbb{C}}}

\newcommand{\ZZ}{\ensuremath{\mathbb{Z}}}
\newcommand{\QQ}{\ensuremath{\mathbb{Q}}}

\newcommand{\HH}{\ensuremath{\mathbb{H}}}

\newcommand{\PP}{\ensuremath{\mathbb{P}}}

\newcommand{\ra}{\ensuremath{\rightarrow}}

\def\eea{\end{eqnarray*}}
\def\bea{\begin{eqnarray*}}

\newcommand\dual{\mathrel{\raise3pt\hbox{$\underline{\mathrm{\thinspace d
\thinspace}}$}}}
\newcommand\qe{\ifhmode\unskip\nobreak\fi\quad $\Box$}       

\def\BOX{\hfill\lower.5\baselineskip\hbox{$\Box$}}

\newtheorem*{mtheorem1}{Main Theorem 1}
\newtheorem*{mtheorem2}{Main Theorem 2}

\newtheorem{theorem}{Theorem}[section]
\newtheorem{lemma}[theorem]{Lemma}
\newtheorem{corollary}[theorem]{Corollary}
\newtheorem{proposition}[theorem]{Proposition}

\theoremstyle{remark}
\newtheorem{remark}[theorem]{Remark}

\theoremstyle{definition}
\newtheorem{definition}[theorem]{Definition}

\DeclareMathOperator{\Id}{Id}
\DeclareMathOperator{\Aut}{Aut}

\DeclareMathOperator{\diag}{diag}

\DeclareMathOperator{\rk}{rk}
\DeclareMathOperator{\GL}{GL}
\DeclareMathOperator{\SL}{SL}
\DeclareMathOperator{\im}{im}
\DeclareMathOperator{\Aff}{Aff}
\DeclareMathOperator{\Irr}{Irr}
\DeclareMathOperator{\Heis}{Heis}

\begin{document}

\author[I. Bauer,  F. Catanese, D. Frapporti]{Ingrid Bauer, Fabrizio Catanese, Davide Frapporti }

\title[Burniat type surfaces and BdF varieties]{Generalized Burniat Type surfaces and Bagnera-de Franchis varieties}

\thanks{The present work took mainly place in the realm of the DFG
Forschergruppe 790 ``Classification of algebraic
surfaces and compact complex manifolds''.\\ The second author also acknowledges support
of the ERC-advanced Grant 340258-TADMICAMT}

\date{\today}

\makeatletter
\renewcommand\theequation{\thesection.\arabic{equation}}
\@addtoreset{equation}{section}
\makeatother

\begin{abstract}
In this article we construct three new  families of surfaces of general type with $p_g=q =0, K^2=6 $,
and  seven new families of surfaces of general type with $p_g=q =1, K^2=6 $, realizing 10 new fundamental groups. We also show that these families correspond to pairwise distinct irreducible connected components
of the Gieseker moduli space of surfaces of general type.

We  achieve this using two different main ingredients.  First we introduce a new class of surfaces, called generalized Burniat type surfaces,
and   we completely classify them (and the connected components of the moduli space containing them). Second, we introduce the notion of Bagnera-de Franchis varieties:
 these are the free quotients of an Abelian variety by a cyclic  group (not consisting only of translations).
For these   we develop some basic results.

\end{abstract}
\maketitle

{\it Dedicated to the memory of Kunihiko Kodaira with  great admiration.}

\section*{Introduction}
The present paper continues, with  new inputs,  a research developed in a series of articles 
(\cite{bacat},   \cite{bcg}, \cite{BC10burniat2}, \cite{BC11burniat1},  
\cite{keumnaie},  \cite{BC12}, \cite{4names}, \cite{BC13burniat3}, \cite{BC13})
 and dedicated to  the discovery of new surfaces of general type with geometric genus
$p_g = 0$, to their   classification, and to the description of their moduli spaces (see the survey article
\cite{survey} for an account of what is known about surfaces wit $p_g=0$, related conjectures and results).

Indeed, in this article,  we consider  the  more general  case of surfaces of general type with $\chi = 1$, i.e., with $p_g = q$.

In the first part we  focus  again on  the construction method originally due to Burniat
(singular bidouble coverings), but 
in  the reformulation done by Inoue (quotients by Abelian groups of exponent two), presenting it in a  rather general fashion which  shows how topological methods allow to describe explicitly connected components of moduli spaces. 
A first novelty   here  is a refined analysis of pencils of Del Pezzo surfaces admitting
a certain group of symmetries, as we shall  now explain.

In a more  general approach  (cf. \cite{BC13}) we consider quotients (cf. \cite{BC12} for  the case of a free action, treated there  in an even  greater generality), by some group $G$ of the form $(\mathbb{Z}/m)^r$, of varieties $\hat{X}$ contained in a product of curves $\Pi_i C_i$, where each $C_i$  is  a maximal Abelian cover
of the projective line with Galois  group of exponent $m$ and with fixed branch locus.

 In the case $m=2$ there is a connection with the Burniat surfaces:
these  are surfaces of general type with invariants $p_g=0$ and $K^2 = 6,5,4,3,2$, whose birational
models  were constructed by Pol Burniat (cf. \cite{burniat}) in 1966 as singular bidouble covers of the projective plane. Later these surfaces were reconstructed by Inoue (cf. \cite{inoue}) as $G:=(\mathbb{Z}/2\mathbb{Z})^3$-quotients of a ($G$-invariant) hypersurface $\hat{X}$ of multidegree $(2,2,2)$ in a product of three elliptic curves.  

While Inoue writes the (affine) equation of $\hat{X}$ in terms of the uniformizing parameters of the respective elliptic curves using a variant of the Weierstrass' function (a Legendre function), we found it much more useful to write the elliptic curves as the complete intersection of two diagonal quadrics in three space.

 This algebraic and  systematic approach  allows us, also with the aid  of computer algebra,  to find all the possible such constructions.
  
Our  situation is as follows: we consider first  the following diagram of quotient morphisms:

\begin{equation}\label{burniatdiagr}
\xymatrix{
E_1 \times E_2 \times E_3 \ar[dd]_{\mathcal{H}':=(\mathbb{Z}/2)^3}^{\pi'}& E_1: = \{x_1^2 + x_2^2 + x_3^2 = 0, \  x_0^2 = a_1x_1^2 + a_2x_2^2\}\\
&E_2: = \{  u_1^2 + u_2^2 + u_3^2 = 0, \ u_0^2 = b_1u_1^2 + b_2u_2^2 \}\\
P_1:=\mathbb{P}^1 \times \mathbb{P}^1 \times \mathbb{P}^1\ar[dd]^{\pi }_{\mathcal{H}:=((\mathbb{Z}/2)^2)^3} & E_3: = \{z_1^2 +z_2^2 +z_3^2 = 0,  \ z_0^2 = c_1 z_1^2 + c_2z_2^2  \}\\
&\\
P_2:=\mathbb{P}^1 \times \mathbb{P}^1 \times \mathbb{P}^1
}
\end{equation}

\noindent 
where the map $\pi'$ is given by ``forgetting'' the variables $x_0,\, u_0,\, z_0$,\\
 the map $\pi$ is given by setting $x_j^2=y_j$, $u_j^2=v_j$, $z_j^2=w_j$,  $j=1,2,3$, and where we view
$P_2\subset (\mathbb P^2)^3$ as the subvariety defined  by the equations
$$y_1+y_2+y_3=0\, ,\, v_1+v_2+v_3=0\, ,\, w_1+w_2+w_3=0\,.$$
The  Galois group for $\pi \circ  \pi'$  is  rather large, it is indeed $(\ZZ /2 \ZZ )^9 \cong \{ \pm 1\}^9$.

We consider then $P_1$ with homogeneous coordinates $((s_1:t_1),(s_2:t_2),(s_3:t_3))$ and for each ${\lambda}:=(\lambda_1, \ldots, \lambda_8)\in \mathbb C^8\setminus \{0\}$ we consider the hypersurface $Y_{\lambda}$ of multidegree $(1,1,1)$ in $P_1$ given by the multihomogeneous equation
\begin{eqnarray}
\lambda_1 s_1s_2s_3+\lambda_2 s_1s_2t_3+\lambda_3 s_1t_2s_3+\lambda_4 s_1t_2t_3+\\
\lambda_5 t_1s_2s_3+\lambda_6 t_1s_2t_3+\lambda_7 t_1t_2s_3+\lambda_8 t_1t_2t_3=0.\nonumber
\end{eqnarray}

We then classify the subgroups $H_1$ (resp. $H_0$) of $\mathcal{H} \cong ((\mathbb{Z}/ 2 \mathbb{Z})^2)^3$
which are  isomorphic to $(\mathbb{Z}/ 2 \mathbb{Z})^2$ (resp. to $(\mathbb{Z}/ 2 \mathbb{Z})^3$) 
and satisfy the property  that there is an irreducible Del Pezzo surface $Y_{\lambda}$ invariant under  $H_1$ 
(resp. $H_0$).

We consider then  $\hat{X}_{\lambda}:= (\pi')^{-1}(Y_{\lambda})$, which  is then invariant under the subgroup $\mathcal G_1 \cong (\mathbb{Z} /2 \mathbb{Z})^5 \subset (\mathbb{Z} /2 \mathbb{Z})^9$ 
inverse image of $H_1$ (resp. $\mathcal G_0 \cong (\mathbb{Z} /2 \mathbb{Z})^6$). We  determine  in this article
all the subgroups $G \cong (\mathbb{Z}/2\mathbb{Z})^3 \subset \mathcal G_1$ (resp. $\mathcal G_0$), having  the property  that $G$ acts freely on 
$\hat{X}_{\lambda}$.

This leads us to introduce a class of surfaces of general type,  described by  the following 
\begin{definition}
Let $G \cong (\mathbb{Z} / 2 \mathbb{Z})^3 \leq \mathcal G_1$  (resp. $\mathcal G_0$) be such that $G$ acts freely on $\hat{X}_{\lambda}$. Then the minimal resolution $S$ of $X_{\lambda}:=\hat{X}_{\lambda} /G$   is called a {\it generalized Burniat type surface}. 
\end{definition}

With the help of the computer algebra system MAGMA (cf. \cite{MAGMA})  we can classify all generalized Burniat type surfaces 
(=GBT surfaces for short) and can prove the following (see Proposition \ref{onefam} and Theorem \ref{fundgroup})
\begin{mtheorem1} \
\begin{enumerate} 
\item There are 16 irreducible families of GBT surfaces. These  have $K^2 = 6$ and $0 \leq p_g=q \leq 3$. The families are  listed 
in Tables \ref{q0} - \ref{q3}, and the dimension of the irreducible family is  4 in cases $\mathcal S_1$ and 
$\mathcal S_2$, and   3 otherwise.
\item  Among the 16 families of generalized Burniat type surfaces four have $p_g=q=0$ (Table \ref{q0}), 
eight have $p_g=q=1$ (Table \ref{q1}), three have $p_g=q=2$ (Table \ref{q2}) and one has $p_g=q=3$ (Table \ref{q3}).
Family $\mathcal S_2$ is the family of primary Burniat surfaces (the one due to Pol Burniat).
\item The fundamental groups of these families are pairwise non isomorphic, except that $\pi_1(S_{11})\cong\pi_1(S_{12})$
and $\pi_1(S_{14})\cong\pi_1(S_{15})$, where $S_j$ is in the family $\mathcal S_j$
\item
The surfaces in the families $\mathcal S_1$, $\mathcal S_3$ and $\mathcal S_4$  realize new (i.e., up to now unknown) fundamental groups of surfaces with $p_g=0, K^2 = 6$,
while the surfaces in the families $\mathcal S_5$-$\mathcal S_{11}$ realize new  fundamental groups for surfaces with $p_g=q=1, K^2 = 6$.

\item
In cases $\mathcal S_1$-$\mathcal S_{10}$, each family of  generalized Burniat type surfaces maps with a generically finite morphism onto
an irreducible connected component of the Gieseker moduli space of surfaces of general type. 
\end{enumerate}
\end{mtheorem1}

We  use indeed  the techniques developed in \cite{BC12} to determine the irreducible connected components of the moduli space
containing the generalized Burniat type surfaces. We do not spell out all the details in the cases 
$\mathcal S_{13}$-$\mathcal S_{16}$,
since the surfaces that we obtain in this way are not new and have already been classified by other authors.

In cases $\mathcal S_1$-$\mathcal S_{10}$ we can apply the general results of  \cite{BC12}  concerning  classical diagonal  Inoue type varieties in order to describe the
connected components of the moduli space containing the  generalized Burniat type surfaces.

 We then show that it is no coincidence that the fundamental groups of  the families $\mathcal S_{11}$ and 
 $\mathcal S_{12}$ in Table \ref{q1} are isomorphic.
These families of surfaces are shown to  be contained in a larger irreducible family, which corresponds to another realization
as Inoue type varieties. This is done via the concept of a Bagnera-de Franchis variety, which we define simply as
the quotient of an Abelian variety $A$ by a nontrivial  finite cyclic group $G$ acting freely on $A$ and not containing any translation.

We obtain in this way  the following theorem

\begin{mtheorem2}

Define  a Sicilian surface to be any
 minimal surface of general type $S$ such that 
\begin{itemize}
\item
$S$ has invariants $K_S^2 = 6$, $p_g(S) =q(S) = 1$,
\item there exists an unramified double cover
$ \hat{S} \rightarrow S$ with $ q (\hat{S}) = 3$,  
\item the Albanese morphism $ \hat{\alpha} \colon  \hat{S}  \rightarrow A = \Alb(\hat{S})$ is birational onto its image $Z$,
a divisor in $A$ with  $ Z^3 = 12$.
\end{itemize}

\noindent
1) Then the canonical model of $\hat{S}$ is isomorphic to $Z$, and the canonical model of $S$ is isomorphic to $Y = Z / (\mathbb{Z}/2 \mathbb{Z})$. 
$Y$  is a divisor in a Bagnera-de Franchis threefold $ X: = A/ G$, where $A = (A_1 \times A_2) / T$, $ G \cong T \cong \mathbb{Z}/2 \mathbb{Z}$,
and where the action is as in (\ref{BCF}).

\noindent
2) Sicilian  surfaces exist, have an irreducible four dimensional moduli space, and their Albanese map $\alpha \colon S \rightarrow A_1 = A_1/ A_1[2]$ has 
general fibre a non hyperelliptic curve of genus $g=3$.

\noindent
3) A GBT surface is a Sicilian surface if and only if it is  in the family $\mathcal S_{11}$ or 
$\mathcal S_{12}$.

\noindent
4) Any surface homotopically equivalent to a Sicilian surface is a Sicilian surface.

\end{mtheorem2}

Indeed, one can replace the above assumption of homotopy equivalence by a weaker one, see Corollary \ref{he}.

In Section \ref{bdf} we discuss the basic results of the theory of Bagnera-de Franchis varieties, and show how
to describe concretely the effective  divisors on them, thus solving in a special  case one of the main technical difficulties in the
general theory of Inoue type varieties, developed in \cite{BC12}.


\section{Inoue's description of Burniat surfaces}\label{Inouedescription}

We briefly recall the description of (primary) \textit{Burniat surfaces} (those  constructed  by P. Burniat in  \cite{burniat}) given by Inoue in \cite{inoue}. 

Inoue considers, for $j\in \{1,2,3\}$,  a complex elliptic curve $E_j:=\mathbb{C}/\langle 1,\tau_j\rangle$ with 
  uniformizing parameter $z_j$, and then   the following three commuting  involutions
on the Abelian variety $A^0:=E_1\times E_2 \times E_3$:

$$\begin{array}{rcc}
g_1(z_1,z_2,z_3)= (-z_1+\frac{1}{2},z_2+\frac{1}{2}, z_3)\,,	 	\\
g_2(z_1,z_2,z_3)= (z_1, -z_2+\frac{1}{2}, z_3+\frac{1}{2})\,, \\
g_3(z_1,z_2,z_3)= (z_1+\frac{1}{2},z_2, -z_3+\frac{1}{2})\,.
\end{array}$$

\noindent 
Note that $G:=\langle g_1, g_2,g_3\rangle\cong (\mathbb Z/ 2\mathbb Z)^3$.

\noindent
Let   $\mathcal{L}_j$, 
for $j=1,2,3$, be a Legendre function for $E_j$: $\mathcal{L}_j\colon E_j\rightarrow \mathbb{P}^1$, a meromorphic function  
which makes $E_j$  a double cover of $\PP^1$ branched over the  four distinct points:
$\pm 1, \pm a_j\in \mathbb{P}^1\setminus \{0, \infty\}$. 

\noindent
It is well known that the following statements hold (see \cite[Lemma 3-2]{inoue} and \cite[Section 1]{BC11burniat1} for an algebraic treatment):
\begin{itemize}
\item $\mathcal{L}_j(0)=1$, $\mathcal{L}_j(\frac{1}{2})=-1$,  $\mathcal{L}_j(\frac{\tau_j}{2})=a_j$,
 $\mathcal{L}_j(\frac{\tau_j+1}{2})=-a_j$;
\item set $b_j:=\mathcal{L}_j(\frac{\tau_j}{4})$: then $b_j^2=a_j$;
\item $\frac{\mathrm d\mathcal L_j}{\mathrm d z_j}(z_j)=0$ if and only if 
$z_j\in\{0,\frac 12, \frac {\tau_j}2, \frac {\tau_j+1}2\}$
 since these are the ramification points of $\mathcal{L}_j$.
\end{itemize}

\noindent
Moreover,
\begin{eqnarray}
&\mathcal{L}_j(z_j)=\mathcal{L}_j(z_j+1)=\mathcal{L}_j(z_j+\tau_j)=\mathcal{L}_j(-z_j)=
-\mathcal{L}_j\bigg(z_j+\dfrac 12\bigg),\nonumber \\
&\mathcal{L}_j\bigg(z_j+\dfrac{\tau_j}2\bigg)= \dfrac{a_j} {\mathcal L_j(z_j)}\,.\nonumber
\end{eqnarray}

For $c \in \mathbb{C}\setminus\{0\}$, Inoue considers the surface 
$$\hat X_c:=\{[z_1,z_2,z_3]\in A^0 \mid \mathcal{L}_1(z_1)\mathcal{L}_2(z_2)\mathcal{L}_3(z_3)=c\}\,$$
inside the Abelian variety $A^0$. Then he shows:
\begin{itemize}

\item $\hat X_c$ is a hypersurface in $A^0$ of multidegree $(2,2,2)$ and is invariant under the action of $G$, $\forall c$.

\item For a general choice of $c$, $\hat X_c$ is smooth, and $G$ acts freely on $\hat X_c$,  
whence $X_c:=\hat X_c/ G$ is a smooth
minimal surface of general type with $p_g=0$ and $K^2=6$.

\item For special values of $c$, the hypersurface $\hat X_c$ has $4,8,12,16$ nodes,
which are isolated fixed points of $G$; in these cases the minimal resolution of singularities of 
$X_c:=\hat X_c/ G$ is a minimal surface of general type with  $p_g=0$ and $K^2=5,4,3,2$.

\end{itemize}

\begin{remark} The minimal resolution of singularities $S_c$ of $X_c$ is called a  {\it Burniat surface}. If $X_c$ is already smooth, or equivalently if $K_{S_c}^2 = 6$, then $S_c$ is called a {\it primary} Burniat surface. For an extensive treatment of Burniat surfaces and their moduli spaces we refer to \cite{BC11burniat1}, \cite{BC10burniat2}, \cite{BC13burniat3}.
\end{remark}

\section{ Intersection of diagonal quadrics and {$(\mathbb Z/ 2\mathbb Z)^n$}-actions}\label{intersection}

As already in \cite[Section 3]{BC13}, we exhibit  $A^0$ as a Galois covering of $(\PP^1)^3$
with Galois group  $ \cong (\ZZ/2)^9$.
This is done  via the following diagram.

 The main purpose of this section is to find irreducible Del Pezzo surfaces in $P_1$ which are left  invariant under 
large subgroups of the group $\sH \cong  (\ZZ/2)^6$.

\begin{equation}\label{diag1}
\xymatrix{
E_1\times E_2\times E_3 \ar[dd]^{\pi'}_{\mathcal H':=(\mathbb Z/ 2\mathbb Z)^3} & E_1:=\{ x_1^2+x_2^2+x_3^2=0,
 \quad x_0^2=a_1x_1^2+a_2x_2^2\}\\
& E_2:= \{u_1^2+u_2^2+u_3^2=0, \quad u_0^2=b_1u_1^2+b_2u_2^2\} \\
P_1:=\mathbb P^1\times\mathbb P^1\times\mathbb P^1    \ar[dd]^{\pi}_{\mathcal H:=((\mathbb Z/ 2\mathbb Z)^2)^3}&
E_3: = \{z_1^2+z_2^2+z_3^2=0, \quad z_0^2=c_1z_1^2+c_2z_2^2\}\\
&\\
P_2:=\mathbb P^1\times\mathbb P^1\times\mathbb P^1\\
}
\end{equation}

\noindent 
The map $\pi'$ is given by ``forgetting'' the variables $x_0,\, u_0,\, z_0$,
whereas the map $\pi$ is given by setting $x_j^2=y_j$, $u_j^2=v_j$, $z_j^2=w_j$,  $j=1,2,3$, and viewing
$P_2\subset (\mathbb P^2)^3$ as the subvariety defined  by the equations
$$y_1+y_2+y_3=0\, ,\, v_1+v_2+v_3=0\, ,\, w_1+w_2+w_3=0\,.$$
The  Galois group for $\pi \circ  \pi'$,  is  $(\ZZ /2 \ZZ )^9 \cong \{ \pm 1\}^9$.

Restricting  diagram (\ref{diag1}) to one (w.l.o.g. the first) factor we get:

\begin{equation}\label{diag2}
\xymatrix{ 
E_1=E \ar[d]_{\mathbb Z/ 2\mathbb Z}&& \\
\mathbb P^1\ar[d]_{(\mathbb Z/ 2\mathbb Z)^2} &= &\{x_1^2+x_2^2+x_3^2=0\}=:C\subset \mathbb P^2\\
\mathbb P^1&= &\{y_1+y_2+y_3=0\}\subset \mathbb P^2
}
\end{equation}

\noindent Since
$$x_1^2+x_2^2+x_3^2=0 \Longleftrightarrow \det \left ( \begin{array}{cc}
x_1+ix_2 & -x_3\\
x_3 & x_1-i x_2
\end{array}\right)=0\,, $$
 we get an isomorphism of $C$ with $\PP^1$:
$$(s:t)= (x_1+ix_2: x_3)= (-x_3: x_1-i x_2)\,$$
and a parametrization of $C$
$$ (x_1 : x_2 : x_3) = (i (s^2 - t^2) :(s^2  + t^2)  : 2 i  st).$$

\noindent 
With this parametrization, we can rewrite the action of $(\mathbb Z/ 2\mathbb Z)^2$ on $\mathbb P^1$ in the following way
 (on the left hand side we use the convenient notation by which all variables not mentioned in a transformation are left unchanged by the transformation):
\begin{itemize}
\item[a)] $x_1\mapsto -x_1$ corresponds to $A_1\colon (s:t)\mapsto(t:s)$;
\item[b)] $x_2\mapsto -x_2$ corresponds to $A_{-1}\colon (s:t)\mapsto(-t:s)$;
\item[c)] $x_3\mapsto -x_3$ corresponds to $B\colon (s:t)\mapsto(s:-t)$. 
\end{itemize}

\noindent  The fixed points of these three involutions are respectively:
\begin{itemize}
\item[a)] $s=\pm t$, equivalently, $ x_1=x_3\pm i x_2=0$;
\item[b)] $s=\pm i t$,  equivalently, $ x_2=x_3\pm i x_1=0$;
\item[c)] $st=0$,  equivalently, $ x_3=x_1\pm i x_2=0$.
\end{itemize}

\

For each ${\lambda}:=(\lambda_1, \ldots, \lambda_8)\in \mathbb C^8\setminus \{0\}$ we consider the hypersurface $Y_{\lambda}$ of multidegree $(1,1,1)$ in $P_1 = \mathbb P^1_{(s_1:t_1)}\times\mathbb P^1_{(s_2:t_2)}\times\mathbb P^1_{(s_3:t_3)} $ given by the multihomogeneous equation
\begin{eqnarray}
\lambda_1 s_1s_2s_3+\lambda_2 s_1s_2t_3+\lambda_3 s_1t_2s_3+\lambda_4 s_1t_2t_3+\\
\lambda_5 t_1s_2s_3+\lambda_6 t_1s_2t_3+\lambda_7 t_1t_2s_3+\lambda_8 t_1t_2t_3=0.\nonumber
\end{eqnarray}

\noindent
Clearly, $Y_{\lambda}$  is a Del Pezzo surface of degree 6.
\noindent
Since we shall be looking for Del Pezzo surfaces  $Y_{\lambda}$ which are left  invariant by certain subgroups of $\mathcal H$
(the Galois group of $\pi$),  we first need to establish conditions ensuring that  the hypersurface $Y_{\lambda}$ is left  invariant by an element $h=(h_1,h_2,h_3)\in \mathcal H$. 

This is done in the next lemma, which is easy to verify and which takes care of the normal form of a transformation $(h_1,h_2,h_3) \in \sH$, taken  up
to a permutation of the three factors
(here $\Id$ is the identity map of $\mathbb P^1$, 
while $A_1$, $A_{-1}$ and $B$ are  the maps
defined above).

\begin{lemma}
Let $h=(h_1,h_2,h_3)\in \mathcal H \setminus \{\Id\}$ be one of the transformations listed in the first column of Table \ref{tab1}. 

Then $Y_{\lambda}$ is $h$-invariant if  and only if the coefficients $\la_j$ satisfy the linear conditions 
 listed in Table \ref{tab1}. 

\end{lemma}
{\small\begin{center}
\begin{table}[!h]
\renewcommand\arraystretch{1}
\setlength{\tabcolsep}{2pt}
\begin{tabular}{c||c|c|c|c|c|c|c|c||c}
 $h$ &  $\lambda_1$ &$\lambda_2$   & $\lambda_3$  &  $\lambda_4$ & $\lambda_5$  &   $\lambda_6$&
 $\lambda_7$   &  $\lambda_8$ &   $c^2$\\ 
\hline
 $\Id,\Id,A_{\alpha_3}$ &   &  $c\lambda_1$ &   & $c\lambda_3$  &   & $c\lambda_5$  &   &  $c\lambda_7$ &  $\alpha_3$ \\ 
\hline 
 \multirow{2}{*}{$\Id,\Id,B$} &   &  0 &   &   0&   &   0&   &   0&   \\ 
 &  0 &   &   0&   &  0 &   &  0 &  &   \\ 
  \hline 
 $\Id,A_{\alpha_2},A_{\alpha_3}$ &   &   & $c\alpha_3\lambda_2$  &  $c\lambda_1$ &   &   &  $c\alpha_3\lambda_6$
  &   $c\lambda_5$&   $\alpha_2\alpha_3$\\ 
\hline 
 $\Id,A_{\alpha_2},B$ &   &   & $c\lambda_1$  & $-c\lambda_2$  &   &   &  $c\lambda_5$ &   $-c\lambda_6$&
   $\alpha_2$ \\ 
\hline 
\multirow{2}{*}{$\Id,B,B$}  &   &  0 &  0 &   &   & 0  &0   &   &   \\ 
        &  0 &   &     &0   & 0  &   &   & 0   &  \\ 
\hline 
$A_{\alpha_1},A_{\alpha_2},A_{\alpha_3}$  &   &   &   &   &  $c\alpha_2\alpha_3\lambda_4$ &  $c\alpha_2\lambda_3$ 
&  $c\alpha_3\lambda_2$ &   $c\lambda_1$&  $\alpha_1\alpha_2\alpha_3$ \\ 
\hline 
$A_{\alpha_1},A_{\alpha_2},B $ &   &   &   &   &   $c\alpha_2\lambda_3$ &$-c\alpha_2\lambda_4$    &$c\lambda_1$    
&$-c\lambda_2$    &$\alpha_1\alpha_2$   \\ 
\hline 
$A_{\alpha_1},B,B$  &   &   &   &  &  $c\lambda_1$&  $-c\lambda_2$ &  $-c\lambda_3$   &$c\lambda_4$   & 	
$\alpha_1$  \\ 
\hline 
\multirow{2}{*}{$B,B,B$}  &   &0   & 0  &   & 0  &   &   &  0  &  \\ 
        &  0 &   &     &0   &   &0   & 0  &    &  \\ 

\end{tabular}
\caption{ } 
\label{tab1}
\end{table}
\end{center}}

\noindent
Note that in Table \ref{tab1},  the numbers $\alpha_i \in \{ \pm 1 \}$, since they are labelling $A_1$ and $A_{-1}$. If for a given case there appear two rows,
this means that there are two alternatives, one for each row.

\begin{remark}		\label{ibb}
Consider the following matrices: 
\begin{equation}
\Gamma_1:=\left(\begin{array}{cc}1&1\\1&-1\\\end{array}\right) \qquad
\Gamma_{-1}:=\left(\begin{array}{cc}i&i\\-1&1\\ \end{array}\right) \,,
\end{equation}
and  denote by $f_1$, respectively $f_{-1}$, the induced projectivities in  $\mathrm{Aut}(\mathbb P^1)$ (observe that $f_1 = f_1^{-1}$).

It is  straightforward  to verify the following conjugacies
\begin{itemize}
\item $B= f_1^{-1} \circ A_1\circ f_1 =  f_{-1} ^{-1} \circ A_{-1} \circ f_{-1}$,
\item $ A_1= f_1^{-1} \circ B\circ f_1 = f_{-1} ^{-1} \circ B\circ f_{-1}$,
\item $A_{-1}= f_1^{-1} \circ A_{-1}\circ f_1 = f_{-1} ^{-1} \circ A_1\circ f_{-1}$.	
\end{itemize}
\end{remark}

\begin{remark}\label{not-irr}

If  $Y_{{\lambda}}$ is invariant under $h=(\Id,\Id,A_\alpha) $ ($\alpha=\pm 1$), or under $h=(\Id,\Id,B) $ then the equation of $Y_{{\lambda}}$ is 
reducible. Since these projectivities are conjugate, it suffices to consider the case $h=(\Id,\Id,B) $, when 
the equation of $Y_{{\lambda}}$ is 
\begin{eqnarray*}
s_3 ( \lambda_1s_1s_2+\lambda_3s_1t_2+\lambda_5t_1s_2+\lambda_7t_1t_2)=0 &\mbox{ or } \\
t_3 ( \lambda_2s_1s_2+\lambda_4s_1t_2+\lambda_6t_1s_2+\lambda_8t_1t_2)=0  
\end{eqnarray*}
 \end{remark}

\noindent The above enable us to  prove the following:
\begin{proposition}\label{cc2}
Let ${\lambda} \in \mathbb C^8 \setminus \{0\}$ be such that $Y_{{\lambda}}$ is irreducible. Assume moreover that there is a subgroup $H_1\cong (\mathbb Z/ 2\mathbb Z)^2$  of $\mathcal H$,
such that $Y_{{\lambda}}$ is $H_1$-invariant.
Then, up to the action of $\PP \GL(2, \CC)^3$ and up to a permutation of the factors of $(\mathbb P^1)^3$, there are exactly  two possibilities:
\begin{itemize}
\item[i)] $H_1=\langle (A_1,A_1,A_1), (\Id,B,B) \rangle$,  or
\item[ii)] $H_1=\langle (\Id,B,B), (B,B,\Id) \rangle$.
\end{itemize}
\end{proposition}

\begin{proof} Let $H_1=\langle h,h' \rangle$. Then,  by  Remarks \ref{not-irr} and \ref{ibb}, after possibly  changing the coordinates of $(\mathbb P^1)^3$,  
we may assume that  $h=(B, B, B) $  or $=(\Id,B,B)$.

\noindent
\underline{1)  $h=(B,B,B)$:} in this case  $h'\in\{ (\Id,B,B), (A_{\alpha_1},B,B)\}$ implies that $(B, \Id, \Id) \in H_1$ or  $(A_{\alpha_1}, \Id, \Id) \in H_1$, contradicting the irreducibility of $Y_{{\lambda}}$ (cf. Remark \ref{not-irr}).

\noindent
If we assume  that $h'\in\{ (\Id,A_{\alpha_2},B), (A_{\alpha_1},A_{\alpha_2},A_{\alpha_3})\}$, $\alpha_i \in \{ \pm 1\}$, then we see (cf. Table \ref{tab1}) that the invariance of  $Y_{{\lambda}}$ under $h$ and $h'$ implies that ${{\lambda}}=0$: this is a contradiction.

\noindent
Assuming instead  that $h'=(\Id,A_{\alpha_2},A_{\alpha_3})$, then conjugating $h'$ by $(f_1, f_{\alpha_2}, f_{\alpha_3})$, we see that in the new coordinates we have:
$$h=(f_1^{-1}\,B\,f_1, f_{\alpha_2}^{-1}\,B \,f_{\alpha_2}, f_{\alpha_3}^{-1}\,B\,
 f_{\alpha_3})=(A_1,A_1,A_1)$$
and
$$h'=(f_1^{-1}\,\Id\,f_1, f_{\alpha_2}^{-1}\,A_{\alpha_2}\, f_{\alpha_2}, f_{\alpha_3}^{-1}\,A_{\alpha_3}\,
 f_{\alpha_3})=(\Id,B,B)\,,$$
 i.e., we are in case i).
 
\noindent
 Assume finally that $h'=(A_{\pm 1},A_{\pm 1}, B)$. Then $h\cdot h' = (A_{\mp 1},A_{\mp 1}, \Id)$ and we reduce to the previous case showing that we are in case i).

\noindent
\underline{2) $h=(\Id,B,B)$:} in this case if $h'= (\Id,A_{\alpha_2},A_{\alpha_3})$, the equation of 
 $Y_{{\lambda}}$ is (cf. Table \ref{tab1}): 
$$
(\lambda_1s_1+\lambda_5t_1)(s_2s_3+ct_2t_3)=0,
$$
contradicting the irreducibility of $Y_{{\lambda}}$.

\noindent
If  $h'\in\{(B,B,B),(\Id,A_{\alpha_2},B), (\Id,B,A_{\alpha_3}),(A_{\alpha_1},B,B)\}$, we obtain that $Y_{{\lambda}}$ is not irreducible by Remark \ref{not-irr}.

\noindent
Assume that $h'\in\{ (A_{\alpha_1},\Id,A_{\alpha_3}), (A_{\alpha_1},A_{\alpha_2},\Id),  (B,A_{\alpha_2},\Id), (B,\Id,A_{\alpha_3}), \\   (A_{\alpha_1},A_{\alpha_2},B), (A_{\alpha_1},B,A_{\alpha_3}),   (B,A_{\alpha_2},B), (B,B,A_{\alpha_3})\}$. Then one checks easily, consulting Table \ref{tab1}, that ${{\lambda}}=0$, hence also these cases can be excluded.

\noindent
If $h'\in\{ (A_{\alpha},\Id,B), (A_{\alpha},B,\Id)\}$, $\alpha \in \{ \pm 1\}$, after changing the coordinates  by $(f_\alpha, \Id,\Id)$  we get $H_1=\langle(\Id,B,B), (B,\Id,B)\rangle$, hence we are in case ii).

\noindent
Assume now that  $h'=(A_{\alpha_1},A_{\alpha_2}, A_{\alpha_3})$. Changing coordinates by conjugating with $(\gamma_1, \gamma_2, \gamma_3)$, where $\gamma_j:= \Id$ if $\alpha_j=1$ and 
 $\gamma_j:=( f_{-1}\circ f_1)$ if $\alpha_j=-1$ and using the fact that
$$
( f_{-1}\circ f_1)^{-1}\circ B\circ ( f_{-1}\circ f_1)=B, \ \  
 ( f_{-1}\circ f_1)^{-1}\circ A_{-1}\circ ( f_{-1}\circ f_1)=A_1, 
$$
 we see that (in the new coordinates) we are in case i).

\noindent
If  $h'=(B,A_{\alpha_2},A_{\alpha_3})$, then changing the coordinates by conjugating with 
 $(f_1, \gamma_2, \gamma_3)$, where $\gamma_j$ is defined as above,   we are in case i).

\noindent
Finally, if $h'\in\{ (B,\Id,B), (B,B,\Id)\}$, then we are in case ii).
\end{proof}

\begin{remark}
It is seen  immediately that in case i) each Del Pezzo surface $Y_{{\lambda}}=\{\lambda_1s_1s_2s_3+\lambda_8 t_1t_2t_3=0\}$ is invariant under $H_1$, whereas in case ii) each surface $Y_{{\lambda}}=\{\lambda_1(s_1s_2s_3+t_1t_2t_3)+ \lambda_4(s_1t_2t_3+t_1s_2s_3)=0\}$ is invariant under $H_1$. In particular, in both respective cases   i) and ii), we obtain a linear action of  $H_1$  on the vector space $V:= H^0((\mathbb P^1)^3, \mathcal O_{(\mathbb P^1)^3} (1,1,1))$, which is independent of the chosen invariant surface in the pencil (see proposition \ref{linearization}).
\end{remark}

\begin{proposition}\label{char1}
With the same notation as in Proposition \ref{cc2}, the respective decompositions of $V$ in character spaces with respect to the above action of $H_1 \cong (\mathbb Z/ 2 \mathbb Z)^2$ are as follows:

\noindent
{\em i)  $H_1=\langle (A_1,A_1,A_1), (\Id,B,B)\rangle$:}
\begin{itemize}
\item $V^{++} = \{\lambda_1(s_1s_2s_3+t_1t_2t_3)+ \lambda_4(s_1t_2t_3+t_1s_2s_3) \mid \lambda_1, \lambda_4 \in \mathbb C \} \cong \mathbb C^2$;
\item $V^{+-} = \{\lambda_2(s_1s_2t_3+t_1t_2s_3)+ \lambda_3(s_1t_2s_3+t_1s_2t_3) \mid \lambda_2, \lambda_3 \in \mathbb C \}$;
\item $V^{-+} = \{\lambda_1(s_1s_2s_3-t_1t_2t_3)+ \lambda_4(s_1t_2t_3-t_1s_2s_3) \mid \lambda_1, \lambda_4 \in \mathbb C \} $;
\item $V^{--} = \{\lambda_2(s_1s_2t_3-t_1t_2s_3)+ \lambda_3(s_1t_2s_3-t_1s_2t_3) \mid \lambda_2, \lambda_3 \in \mathbb C \}$.
\end{itemize}

\noindent
{\em ii) $H_1=\langle (\Id,B,B), (B,B,\Id)\rangle$:}
\begin{itemize}
\item $V^{++} = \{\lambda_1s_1s_2s_3+\lambda_8 t_1t_2t_3 \mid \lambda_1, \lambda_8 \in \mathbb C \} \cong \mathbb C^2$;
\item $V^{+-} = \{\lambda_4s_1t_2t_3+\lambda_5 t_1s_2s_3 \mid \lambda_4, \lambda_5 \in \mathbb C \}$;
\item $V^{-+} = \{\lambda_2s_1s_2t_3+\lambda_7 t_1t_2s_3 \mid \lambda_2, \lambda_7 \in \mathbb C \} $;
\item $V^{--} = \{\lambda_3s_1t_2s_3+\lambda_6 t_1s_2t_3 \mid \lambda_3, \lambda_6 \in \mathbb C \}$.
\end{itemize}

\end{proposition}

\begin{proof}
This is a simple calculation using Table \ref{tab1}.
\end{proof}
The same arguments as in the  proof of Proposition \ref{cc2} yield the following statement:

\begin{proposition}\label{cc3}
Let ${\lambda} \in \mathbb C^8 \setminus \{0\}$ be such that $Y_{{\lambda}}$ is irreducible. Assume moreover that there is a subgroup $H_0 \cong (\mathbb Z/ 2\mathbb Z)^3$  of $\mathcal H$,
such that $Y_{{\lambda}}$ is $H_0$-invariant.
Then, up to the action of $\PP \GL(2, \CC)^3$ and up to a permutation of the factors of $(\mathbb P^1)^3$, we have: 
$$H_0=\langle (\Id,B,B),(A_1,A_1,A_1),  (B,B,\Id)\rangle \,.$$
\end{proposition}

\begin{remark}
Again we see immediately that the Del Pezzo surface $Y_{{\lambda}}=\{ s_1s_2s_3 +  t_1t_2t_3\}$ is invariant under $H_0$, hence we get again
a linear action of  $H_0$  on the vector space $V:= H^0((\mathbb P^1)^3, \mathcal O_{(\mathbb P^1)^3} (1,1,1))$.
\end{remark}

\begin{proposition}\label{char2}
Use the same notation as in Proposition \ref{cc3}; then  $V$ decomposes  in $8$ one-dimensional character spaces for the action of $H_0 \cong (\mathbb Z/ 2 \mathbb Z)^3$, as follows:
\begin{itemize}
\item $V^{+++} = \{\lambda(s_1s_2s_3+ t_1t_2t_3 ) \mid \lambda \in \mathbb C \} $;
\item $V^{+-+} = \{\lambda(s_1s_2s_3- t_1t_2t_3 ) \mid \lambda \in \mathbb C \}$;
\item $V^{++-} = \{\lambda(s_1t_2t_3+ t_1s_2s_3 ) \mid \lambda \in \mathbb C \}$;
\item $V^{+--} = \{\lambda(s_1t_2t_3- t_1s_2s_3 ) \mid \lambda \in \mathbb C \}$;
\item $V^{-++} = \{\lambda(s_1s_2t_3+ t_1t_2s_3 ) \mid \lambda \in \mathbb C \}$;
\item $V^{--+} =  \{\lambda(s_1s_2t_3- t_1t_2s_3 ) \mid \lambda \in \mathbb C \}$;
\item $V^{-+-} = \{\lambda(t_1s_2t_3+ s_1t_2s_3 ) \mid \lambda \in \mathbb C \}$;
\item $V^{---} = \{\lambda(t_1s_2t_3- s_1t_2s_3 ) \mid \lambda \in \mathbb C \}$.
\end{itemize}
\end{proposition}

\begin{remark}\label{invDP}
Case 1):
$$H_1: =\langle (\Id,B,B),(A_1,A_1,A_1) \rangle  \cong (\mathbb Z/ 2\mathbb Z)^2\ \triangleleft\mathcal H.$$
Then there are four pencils of Del Pezzo surfaces, which are left  invariant by $H_1$
(cf. Proposition \ref{char1});
 their inverse images under $\pi'$ (see (\ref{diag1}))  
$\pi'^{-1}(Y_\nu)$ (resp. $\pi'^{-1}(Y'_\nu)$,   $\pi'^{-1}(Y''_\nu)$, $\pi'^{-1}(Y'''_\nu)$) 
are  pencils of hypersurfaces of multidegree $(2,2,2)$ in $A^0= E_1 \times E_2 \times E_3$ invariant under
 $\mathcal G'_1\cong (\mathbb Z/ 2\mathbb Z)^5\subset (\mathbb Z/ 2\mathbb Z)^9$. 
 
\noindent We list now the four pencils  (${\nu}=(\nu_1:\nu_2) \in \mathbb P^1$):
\begin{equation}\label{eqDp1}
Y_{\nu}:=\{\nu_1(s_1s_2s_3+ t_1t_2t_3)+ \nu_2(s_1t_2t_3+t_1s_2s_3)=0\}\,, 
\end{equation}
\begin{equation}
Y'_{\nu}:=\{\nu_1(s_1s_2t_3+ t_1t_2s_3)+ \nu_2(s_1t_2s_3+t_1s_2t_3)=0\}\,, 
\end{equation}
\begin{equation}
Y''_{\nu}:=\{\nu_1(s_1s_2s_3- t_1t_2t_3)+ \nu_2(s_1t_2t_3-t_1s_2s_3)=0\}\,, 
\end{equation}
\begin{equation}
Y'''_{\nu}:=\{\nu_1(s_1s_2t_3- t_1t_2s_3)+ \nu_2(s_1t_2s_3-t_1s_2t_3)=0\}\,. 
\end{equation}

It is immediate to see that the $4$ pencils are transformed to each other by the elements 
of the group $ \sH  = ((\ZZ/2)^2)^3$ (for instance we pass from the first to the second via $s_3 \leftrightarrow t_3$,
from the first to the third via $t_1  \leftrightarrow - t_1$, and so on).
Therefore, in the future we shall only consider the first pencil: (\ref{eqDp1}).

Case 2):   $$H_1: =\langle (\Id,B,B), (B,B,\Id)\rangle \cong (\mathbb Z/ 2\mathbb Z)^2\ \triangleleft\mathcal H.$$
Then there are four pencils of Del Pezzo surfaces, which are left  invariant by $H_1$; 
 their respective inverse images  under $\pi'$ yield four pencils, invariant under 
 $\mathcal G_1\cong (\mathbb Z/ 2\mathbb Z)^5\subset (\mathbb Z/ 2\mathbb Z)^9$.
 
   \noindent The four pencils are given by the following equations  ($\mu\in \mathbb C, \mu \neq 0$):
\begin{equation}\label{eqDp2}
Y_\mu:=\{s_1s_2s_3+\mu\, t_1t_2t_3=0\}\,,
\end{equation}
 \begin{equation}
Y'_\mu:=\{s_1t_2s_3+\mu\, t_1s_2t_3=0\}\,,
\end{equation}
\begin{equation}
Y''_\mu:=\{s_1t_2t_3+\mu\, t_1s_2s_3=0\}\,,
\end{equation}
\begin{equation}
Y'''_\mu:=\{s_1s_2t_3+\mu\, t_1t_2s_3=0\}\,.
\end{equation}

Also in this case the $4$ pencils are transformed to each other by the elements 
of the group $ \sH  = ((\ZZ/2)^2)^3$, hence
 in the future we shall only consider the first pencil: (\ref{eqDp2}).

Case 3): 
$$H_0: =\langle (\Id,B,B), (A_1,A_1,A_1),  (B,B,\Id)\rangle \cong ((\mathbb Z/ 2\mathbb Z))^3 \triangleleft\mathcal H.$$ 
Then there are eight Del Pezzo surfaces which are are left  invariant by $H_0$; 
their respective inverse images under $\pi'$ are invariant under 
$\mathcal G_0\cong (\mathbb Z/ 2\mathbb Z)^6\subset (\mathbb Z/ 2\mathbb Z)^9$. 

\noindent Their respective equations are  the following ones:
\begin{equation}\label{eqDp3}
Y_{1}:=\{s_1s_2s_3+  t_1t_2t_3 = 0\}, \ \  Y_{-1}:=\{s_1s_2s_3-  t_1t_2t_3 = 0\},
\end{equation}
\begin{equation}
Y_{1}':=\{s_1t_2t_3+  t_1s_2s_3 = 0\}, \ \  Y_{-1}':=\{s_1t_2t_3-  t_1s_2s_3 = 0\},
\end{equation}
\begin{equation}
Y_{1}'':=\{s_1s_2t_3+  t_1t_2s_3 = 0\}, \ \  Y_{-1}'':=\{s_1s_2t_3-  t_1t_2s_3 = 0\},
\end{equation}
\begin{equation}
Y_{1}''':=\{t_1s_2t_3+  s_1t_2s_3 = 0\}, \ \  Y_{-1}''':=\{t_1s_2t_3-  s_1t_2s_3 = 0\}.
\end{equation}

Also here the  $8$ hypersurfaces  are transformed to each other by the elements 
of the group $ \sH  = ((\ZZ/2)^2)^3$, hence
 in the future we shall only consider the first one: (\ref{eqDp3}).

\end{remark}

\begin{definition}\label{BurniatHyp}
Let $\hat{X}$ be an irreducible hypersurface, in the product of three smooth elliptic curves 
$A^0 := E_1 \times E_2 \times E_3$, which is the 
inverse image under $\pi'$ of  a Del Pezzo surface $Y$ of degree 6,  invariant under a subgroup $H \cong (\mathbb Z/ 2\mathbb Z)^2\ \triangleleft\mathcal H$.\\  Then we call $\hat X$ a  \textit{Burniat hypersurface in $A^0$}.
\end{definition}

\begin{lemma} \label{stabs}
Let $\mathcal G_0\cong (\mathbb Z/ 2\mathbb Z)^6\triangleleft (\mathbb Z/ 2\mathbb Z)^3\times
(\mathbb Z/ 2\mathbb Z)^3\times(\mathbb Z/ 2\mathbb Z)^3$
be the  group:
 $$\mathcal G_0:=\{(\epsilon_0,\eta_1,\epsilon_1,\eta_0,\epsilon_2,\zeta_0,\epsilon_3) 
\subset \{\pm 1\}^7 \cong (\mathbb Z/ 2\mathbb Z)^7 \mid \epsilon_1\epsilon_2\epsilon_3=1\}\,,$$
which acts on $E_1\times E_2\times E_3$ by:
$$\begin{array}{lll}x_0\mapsto \epsilon_0x_0\,,& u_0 \mapsto \eta_0 u_0\,, 
 &z_0\mapsto \zeta_0 z_0\,, \\
x_3\mapsto \epsilon_1 x_3\,,  &u_3 \mapsto \epsilon_2 u_3\,,  &  
z_3 \mapsto \epsilon_3 z_3\,
\end{array}\quad
\mbox{ and }\quad
 \begin{pmatrix}x_1\\u_1\\z_1\end{pmatrix}\mapsto \eta_1
 \begin{pmatrix}x_1\\u_1\\z_1\end{pmatrix}\,.
$$

\noindent With the same notation as in Remark \ref{invDP}:
\begin{enumerate}
\item   $\pi'^{-1}(Y_\nu)$ is invariant under the group 
$$\mathcal G'_1:=\{(\epsilon_0,\eta_1,1,\eta_0,\epsilon_2,\zeta_0,\epsilon_3)  \mid \epsilon_2\epsilon_3=1\}\cong (\mathbb Z/ 2\mathbb Z)^5\triangleleft \mathcal G_0\,.$$
\item  $\pi'^{-1}(Y_\mu)$ is invariant under the group 
$$\mathcal G_1:=\{(\epsilon_0,1,\epsilon_1,\eta_0,\epsilon_2,\zeta_0,\epsilon_3)  \mid
 \epsilon_1\epsilon_2\epsilon_3=1\}\cong (\mathbb Z/ 2\mathbb Z)^5\triangleleft \mathcal G_0\,.$$
\item  If $\mu=\pm 1$, then $\pi'^{-1}(Y_{\mu})$ is invariant under $\mathcal G_0$. 
\end{enumerate}
\end{lemma}

\begin{proof}
Just note that  multiplication of $(x_1,u_1,z_1)$ by $-1$ corresponds to $(s_j:t_j) \mapsto(t_j:s_j)$ for each $j=1,2,3$.
\end{proof}

\subsection{Fixed points}

In order to systematically search for all the subgroups $G\cong  (\mathbb Z/ 2\mathbb Z)^3\triangleleft \mathcal G_0$
acting freely on a Burniat hypersurface in $A^0:=E_1\times E_2\times E_3$, we need to determine which elements of  $\mathcal G_0$ have fixed points on $A^0$ .

\begin{remark}
Fix $a_1,a_2 \in \mathbb C$ pairwise distinct such that the curve
$$ E:=\{x_1^2+x_2^2+x_3^3=0, \, x_0^2=a_1x_1^2+a_2x_2^2 \}\subset \mathbb P^2$$
is smooth. Then
$$g(x_0:x_1:x_2:x_3):=(\alpha_0x_0:  \alpha_1x_1:x_2:\alpha_3x_3)\,, \quad \alpha_j\in \{\pm 1\}\,$$
has fixed points on $E$ if and only if
\begin{itemize}
\item either $\alpha_0=\alpha_1=\alpha_3=-1$, or
\item  exactly one $\alpha_i=-1$ and the others are equal to 1.
\end{itemize} 
\end{remark}

 From now on, we change to an additive notation in which $\mathbb Z /2 \mathbb Z$ is the additive group $\{0,1\}$.

Let $g\in \mathcal G_0$ be an element fixing points on $ A^0$.
By \cite[Proposition 3.3]{BC13}, $g$ is an element in Table \ref{FixEl}. 

\begin{table}[!h]
\begin{tabular}{c||ccc||cccccc||cccccccc}
&1&2&3&4&5&6&7&8&9&10&11&12&13&14&15&16&17\\
\hline 
$\epsilon_0$ &  $ 0 $ & $ 0 $& $ 1 $&$ 0 $ &$ 1 $& $ 1 $&$ 0 $ & $ 0 $ & $ 0 $&$ 1 $ &$ 1 $&$ 0 $ &  $0$& 
$0$ & $ 0 $& $ 1 $& $ 1 $ 	\\
$\eta_1 $&      	$ 0 $ & $ 0 $& $ 0 $& $ 0 $& $ 0 $&$ 0 $ & $ 0 $&$ 0 $ & $ 0 $&$ 0 $ &$ 0 $&$ 0 $ & $ 0 $&
$ 1 $ & $ 1 $& $ 1 $&$ 1 $\\
$\epsilon_1 $&  $ 0 $ & $ 0 $& $ 0 $& $ 0 $& $ 0 $&$ 0 $ & $ 0 $&$ 1 $ & $ 1 $&$ 0 $ &$ 0 $&$ 1 $ & $ 1 $&
$ 0 $ & $ 0 $& $ 1 $&$ 1 $\\
$\eta_0$&        	$ 0 $ & $ 1 $& $ 0 $& $ 1 $& $ 0 $& $ 1 $& $ 0 $&$ 0 $ & $ 0 $&$ 1 $ &$ 0 $&$ 1 $ & $ 0 $&
$ 0 $ & $ 1 $& $ 0 $&$ 1 $\\
$\epsilon_2$& 	$ 0 $ & $ 0 $& $ 0 $& $ 0 $& $ 0 $& $ 0 $& $ 1 $&$ 0 $ & $ 1 $&$ 0 $ &$ 1 $& $ 0 $& $ 1 $&
$ 0 $ & $ 1 $& $ 0 $&$ 1 $\\
$\zeta_0$& 			$ 1 $ & $ 0 $& $ 0 $& $ 1 $& $ 1 $& $ 0 $& $ 0 $&$ 0 $ & $ 0 $& $ 1 $&$ 0 $& $ 0 $& $ 1 $&
$ 0 $ & $ 1 $& $ 1 $&$ 0 $\\
$\epsilon_3$&	$ 0 $ & $ 0 $& $ 0 $& $ 0 $& $ 0 $& $ 0 $& $ 1 $&$ 1 $ & $ 0 $&$ 0 $ &$ 1 $& $ 1 $& $ 0 $&
$ 0 $ & $ 1 $& $ 1 $&$ 0 $\\
\end{tabular}
\vspace{.2cm}
\caption{The elements of $\mathcal G_0$ having fixed points on $A^0$, written additively!}
\label{FixEl}
\end{table}

\begin{remark}\label{fixingpts}
1) Let $\hat X:=\pi'^{-1}(Y_{\pm 1})$. In Table \ref{FixEl}, the elements 1-3 fix pointwise a surface $S\subset A^0$.
Each element 4-9 fixes pointwise a curve $C\subset A^0$ and its fixed locus has non trivial 
intersection with $\hat X$ since $\hat X\subset A^0$ is an ample divisor.
Finally, the elements 10-17  have isolated fixed points on $A^0$; arguing as in \cite[Proposition 3.3]{BC13}
one proves that the elements 11-17 have fixed points on $\hat X$,
 while the fixed locus of element 10 
intersects $\hat X$ only for special choices of the three elliptic curves.

2) The same holds for $\hat X:=\pi'^{-1}(Y_{\nu})$ (resp. $\pi'^{-1}(Y_\mu)$), 
considering only the elements  1-7,10,11,14,15 (resp. 1-13),
i.e. the ones belonging to $\mathcal G'_1$ (resp. $\mathcal G_1$).
In particular, the fixed locus of  element 10 intersects $\hat X$ only for special choices of the three elliptic curves and of  the parameter $\nu$ (resp. $\mu$).
\end{remark}

\subsection{ Description in terms of Legendre families}
We now describe the  families of Burniat hypersurfaces in $A^0$
 in terms of Legendre functions $\mathcal{L}$ (see Section \ref{Inouedescription}).
 
 To this purpose, we consider the following 1-parameter family   of intersections of two quadrics:
 
  $$E  (b) :=\{ x_1^2+x_2^2+x_3^2=0,
 \quad x_0^2= (b^2+1)^2 x_1^2+( b^2-1)^2x_2^2\}\\,$$
 where $ b \in \CC \setminus \{ 0, 1, -1 , i, - i \}$.
 
 We set $$ \xi : = \frac{bs}{t}\,,$$
 and in this way  the family of genus one curves $E(b)$ is the Legendre family of
 elliptic curves in Legendre normal affine form:
 $$ y^2 = (\xi^2-1) (\xi^2-a^2), \ \  a : = b^2\,.$$ 

 \noindent  In fact,  
\begin{eqnarray*} 
  x_0^2&=& (b^2+1)^2 x_1^2+( b^2-1)^2x_2^2 =- (a+1)^2 (s^2-t^2)^2+( a -1)^2(s^2 + t^2) ^2 = \\
 &=& 4 [(a^2 + 1) s^2 t^2  - a (t^4 + s^4)]   = 4 t^4 \left[  (a^2 + 1)\left( \frac{\xi} {b}\right)^2 - a \left( 1 + \left( \frac{\xi} {b}\right)^4\right)\right ] = \\
 & =& - 4 t^4 \frac{1}{b^2} [ - (a^2 + 1) \xi^2  + (a^2 + \xi^4) =  \frac{- 4 t^4}{b^2} [   (\xi^2-1) (\xi^2-a^2)] 
 \end{eqnarray*}
 and it suffices to set $$ y : =  \frac{i b x_0}{2 t^2}\,.$$ 
 
 The group $(\ZZ/ 2)^3$ acts fibrewise on the family $E(b)$ via the commuting involutions:
 $$ x_0 \longleftrightarrow - x_0,  \ \ x_3 \longleftrightarrow  - x_3,\ \  x_1 \longleftrightarrow  - x_1, $$
 which on the birational model given by the Legendre family act as
 $$ y  \longleftrightarrow  -  y,  \ \ \xi  \longleftrightarrow  - \xi, \ \ \xi  \longleftrightarrow  \frac{a}{\xi}\,. $$
 
Consider the subgroup
 $$  \Ga_{2,4} : = \left\{ \left( \begin{array}{cc} \alpha & \beta \\ \gamma & \delta \end{array}  \right)   
\in \PP SL (2 , \ZZ)  \biggm|  \begin{array}{cc}
  \alpha \equiv 1 \mod 4, & \beta \equiv  0 \mod 4, \\\gamma \equiv  0 \mod 2, &\delta \equiv  1 \mod 2
  \end{array}  
\right\} $$
a subgroup of index $2$ of the congruence subgroup  
 $$  \Ga_2 : = \left\{ \left( \begin{array}{cc} \alpha & \beta \\ \gamma & \delta \end{array}  \right)   
 \in \PP SL (2 , \ZZ) \biggm| \begin{array}{cc}
  \alpha \equiv 1 \mod 2, & \beta \equiv  0 \mod 2, \\\gamma \equiv  0 \mod 2, &\delta \equiv  1 \mod 2
  \end{array}  
\right\} .$$

To the chain of inclusions $$ \Ga_{2,4}  <    \Ga_2  < \PP SL (2 , \ZZ) $$
corresponds a chain of fields of invariants 
$$ \CC (j) \subset \CC(\la) = \CC(\tau)^{\Ga_2} \subset  \CC(\tau)^{\Ga_{2,4}}\,, $$
where the respective degrees of the extensions are 6, 2.

Here, $\la$ is the cross-ratio of the four points $\mathfrak p (0), \mathfrak p (\frac{1}{2}), \mathfrak p (\frac{\tau}{2}), \mathfrak p (\frac{1 + \tau}{2}), $
 where $\mathfrak p $ is the Weierstrass function, and $j (\la) = \frac{(\la^2 - \la + 1)^3}{\la^2 (\la-1)^2}$ is the $j$-invariant.

If $\la (a)$ is the cross ratio of the four points $1,-1, a, - a$,    $\la (a) = \frac{(a-1)^2}{( a+1)^2 }$,
thus $\CC(a) = \CC(\sqrt \la)$ is a quadratic extension and there are two values of $a$ for which we get a Legendre function for the elliptic curve.
Setting $b : = \sL (\frac{\tau}{4} )$, we have that 
$a = b^2$, hence $\CC(b)$ is a quadratic extension of  $\CC(\tau)^{\Ga_{2,4}}$.

In other words, the parameter $ b \in \CC \setminus \{ 0, 1, -1 , i, - i \}$  yields an unramified covering of degree $4$ of $ \la  \in \CC \setminus \{ 0, 1\}$ ,
hence the field $\CC(b)$ is the invariant field for a subgroup {{} $\Ga_{2,8}$}   
 of index $2$ in $\Ga_{2,4}$.

{{}
\noindent
By \cite[\S 182]{Bianchi}, $b$ is invariant under the subgroup of $\Ga_2$ given by the transformation
such that $\alpha^2+\alpha\beta \equiv 1 \mod 8$.
Since $\alpha \equiv 1 \mod 2$, this equation is equivalent to require that $\beta \equiv 0 \mod 8$, i.e.:
$$  \Ga_{2,8} : = \left\{ \left( \begin{array}{cc} \alpha & \beta \\ \gamma & \delta \end{array}  \right)   
\in \PP SL (2 , \ZZ)  \biggm|  \begin{array}{cc}
  \alpha \equiv 1 \mod 4, & \beta \equiv  0 \mod 8, \\\gamma \equiv  0 \mod 2, &\delta \equiv  1 \mod 2
  \end{array}  
\right\} .$$
}

 Consider now  the following family 
 $$\sA^0 = E(b_1) \times E(b_2) \times E(b_3).$$ It is the family of  products of three elliptic curves with a  {{} $\Ga_{2,8}$}-level  structure:
 $\sA^0$ is  the quotient of $(\CC\times \HH)^3$,
 with coordinates $((z_1, \tau_1),(z_2, \tau_2),(z_3, \tau_3))$, by the action of the group (a semidirect product ) generated by  
 $(\ZZ^2)^3$ which acts  by 
 $$((m_1, n_1),(m_2, n_2),(m_3, n_3)) \circ ((z_1, \tau_1),(z_2, \tau_2),(z_3, \tau_3))= $$ $$= ((z_1 + m_1 + n_1 \tau_1, \tau_1),(z_2 + m_2 + n_2 \tau_2, \tau_2),(z_3+ m_3 + n_3 \tau_3, \tau_3))$$
 and by $ {{} \Ga_{2,8}} ^3 \subset  \PP SL (2 , \ZZ)^3$.
 
\noindent The fibre of $f : \sA^0 \ra \sE : = ( \HH)^3 /  {{} \Ga_{2,8}} ^3)$ is the product of the  three elliptic curves, 
for $k= 1,2,3$, $E_k: =\mathbb C/\langle 1, \tau_k\rangle$.

Let $\mathcal{L}_k\colon E_k\rightarrow \mathbb P^1$ be a Legendre function for $E_k$.
We have seen that the relation between $\mathcal L_k( z_k)$ and the coordinates $(s_k:t_k)$ of $\mathbb P^1$
is $$\dfrac{\mathcal{L}_k(z_k)}{b_k}= \dfrac{s_k}{t_k}$$
where $b_k:=\mathcal L_k(\frac{\tau_k}{4})$.
A basis for the $(\mathbb{Z}/2 \mathbb{Z})^3$-action on $E_k$, $k=1,2,3$,  is given by:
\begin{equation}\label{Z23-action}
\begin{array}{ccccc}
 x_0  \mapsto - x_0 &\hat{=} & (z_k\mapsto -z_k)&\hat{=}&(1,0,0)\\
x_1 \mapsto  - x_1 &\hat{=} & (z_k\mapsto -z_k+\frac{\tau_k}{2})&\hat{=}&(0,1,0)\\
x_3 \mapsto  - x_3 &\hat{=} & (z_k\mapsto -z_k+\frac{1}{2})&\hat{=}&(0,0,1)
\end{array}
\end{equation}

The above formulae  define an action of $((\ZZ/2)^3)^3$ on the fibration  $f : \sA^0 \ra \sE$, which acts trivially on the basis.

It follows from  (\ref{eqDp1}, \ref{eqDp2}, \ref{eqDp3}) that  it suffices to consider only  the families of Burniat hypersurfaces  defined by:
\begin{equation}\label{righteq}
\begin{array}{l}
\hat \sX_\nu= \{([(z_1, \tau_1),(z_2, \tau_2),(z_3, \tau_3)], (\nu_1 : \nu_2))  \in \sA^0 \times \PP^1 \mid\\ \\
\nu_1(\mathcal L_1(z_1)\mathcal L_2(z_2)\mathcal L_3(z_3) +  b_1b_2b_3)+

\nu_2(\mathcal L_1(z_1) b_2b_3 +  b_1\mathcal L_2(z_2)\mathcal L_3(z_3))=0 \}	\,,
\end{array}
\end{equation}
\begin{equation}
\hat \sX_{\mu}= \{([(z_1, \tau_1),(z_2, \tau_2),(z_3, \tau_3)], \mu ) \in \sA^0 \times \CC^* \mid \mathcal L_1(z_1)\mathcal L_2(z_2)\mathcal L_3(z_3)=\mu \}\,,
\end{equation}
\begin{equation}
\hat \sX_b= \{[(z_1, \tau_1),(z_2, \tau_2),(z_3, \tau_3)] \in \sA^0\mid \mathcal L_1(z_1)\mathcal L_2(z_2)\mathcal L_3(z_3)=b_1b_2b_3\}\,,
\end{equation}

\noindent where the meaning of the subscript is to refer to the variables:
${\nu}=(\nu_1:\nu_2) \in \mathbb P^1 , \ \mu\in \CC^*, \ \ b:=b_1b_2b_3$.

\begin{remark}\label{1fam}
There is also an obvious action of the symmetric group $\mathfrak S_3$ on the family $f : \sA^0 \ra \sE$.

\end{remark}

Let  $\hat X$   
 be a Burniat hypersurface in $A^0$ (see Definition \ref{BurniatHyp}).
An explicit calculation using the above equations shows that  $\hat X$ has at most finitely many nodes as singularities.

Let $\epsilon\colon X' \rightarrow \hat X$ be the minimal resolution of its singularities.
Since $\hat{X}$ has at most canonical singularities, $K_{X'}=\epsilon^* K_{\hat X}$ and 
 $X'$ is a minimal surface of general type with $K^2_{X'}=48$ and $\chi(X')=8$ (cf. \cite{inoue}).

\section{Generalized Burniat type surfaces}

Using the notation introduced in the previous sections, we give the following definition.
\begin{definition}
Let $\hat X$ be a  Burniat hypersurface in $A^0:=E_1 \times E_2 \times E_3$, let  $G\cong  (\mathbb Z/ 2\mathbb Z)^3$ be 
a subgroup of $\mathcal G_0$ 
 acting  freely on $\hat X$.

The minimal resolution $S$ of the quotient surface $ X:= \hat X/G$ is called a 
\textit{generalized Burniat  type (GBT) surface}. We call $ X$ the \textit{quotient model of $S$} (indeed, we easily see that
$X$ is the canonical model of $S$).
\end{definition}

\begin{remark}1)
Since $G$ acts freely and $\hat X$ has at most nodes as singularities 
(we assume $Y$, hence also $\hat{X}$, to be irreducible!), 
a generalized Burniat type surface $S$ is a smooth minimal surface of general type with
$K^2_S=6$ and $\chi(S)=1$.

2) If $G\triangleleft \mathcal G_1 $ or  $G\triangleleft \mathcal G'_1 $, then there is a pencil of Burniat hypersurfaces which are left 
invariant by the $G$-action,  and the family of quotients of the hypersurfaces on which the action is free  is then a 
one parameter family of GBT $G$-quotient surfaces (if we vary also  $E_1 , E_2 ,E_3$ we obtain a four dimensional family).
\end{remark}

\noindent
Let $\De$ be the subgroup of $\Aut(( (\mathbb Z/ 2\mathbb Z)^3)^3  ) $ generated by:\\
\begin{minipage}{.5\textwidth}
\begin{eqnarray*}
l_1(g_1,g_2,g_3) &=&  (g_2,g_1,g_3)\\
l_2(g_1,g_2,g_3) &=& (g_3,g_2,g_1)
\end{eqnarray*}

\end{minipage}
\begin{minipage}{.5\textwidth}
\begin{eqnarray*}
h_1(g_1,g_2,g_3) &=& (f(g_1), f(g_2), g_3)\\
h_2(g_1,g_2,g_3) &=& (f(g_1), g_2, f(g_3))\\
h_3(g_1,g_2,g_3) &=& (g_1, f(g_2), f(g_3))
\end{eqnarray*}

\end{minipage}
\vspace{.1cm}

\noindent where $g_j \in (\mathbb Z/ 2\mathbb Z)^3 \, (j\in\{1,2,3\})$ and where $f$ is defined by:
  $$\begin{array}{ccc}
f\colon (\mathbb Z/ 2\mathbb Z)^3&\longrightarrow &(\mathbb Z/ 2\mathbb Z)^3   \\[6pt]
f\colon (a,b,c)& \longmapsto &(a+b,b,b+c)
\end{array}$$
\begin{remark} 

1) It is easy to see that $\De (\mathcal{G}_0) = (\mathcal{G}_0)$.

\noindent 
2) We claim now that, as it can be verified, for each $\de \in \De$, $\de (g)$ is    conjugate to $g$ 
via an element $\phi$ of   the group of automorphisms 
of  $\sA^0$.

For example, let  $E:=\mathbb C/\langle 1, \tau \rangle$ be a complex elliptic curve and let $\tau':=\tau+1$.
Then $E=\mathbb C/\langle 1, \tau'\rangle$ and  
the $(\mathbb Z/ 2\mathbb Z)^3$-action, defined in (\ref{Z23-action}), is:
\begin{equation}
\begin{array}{ccc}
(z\mapsto -z)&=&(1,0,0)\\
(z\mapsto -z+\frac{\tau'}{2})=(z\mapsto -z+\frac{\tau+1}{2})&=& (1,1,1)\\
(z\mapsto -z+\frac{1}{2})&=&(0,0,1)
\end{array}
\end{equation}
This shows that the groups $G$ and  $G':=h_j(G)\subset \mathcal G_0\,, j=1,2,3$ are conjugate via an automorphism of $\sA^0$. 

 It follows then 
 that $g\in \mathcal G_0$ acts freely on one of the families 
$\hat \sX $ if and only if $\de (g)$ acts freely on the  transformed family 
$\phi (\hat \sX )$. 

It follows also that two groups in the same $\De$-orbit yield isomorphic  families of GBT surfaces,
hence we can restrict our attention to a single representative for each  $\De$-orbit.
\end{remark}

\begin{proposition}\label{onefam}
\begin{enumerate}
\item There are exactly 16 irreducible families of generalized Burniat type surfaces, listed in tables \ref{q0}-\ref{q3}.
\item The family  of generalized Burniat type surfaces has dimension 4 in cases $\mathcal S_1$ and $\mathcal S_2$,  and dimension 3 otherwise.

\end{enumerate}
\end{proposition}

\begin{proof}
1) The  MAGMA script below searches for subgroups of $G \leq \mathcal{G}_0$, which satisfy the following
\begin{itemize}
\item $G \cong (\ZZ / 2\ZZ)^3$;
\item $G$ does not contain the elements 1-9, 11-17 of table \ref{FixEl}.
\end{itemize}
The 161 groups of the output therefore act freely on $\hat{X}_b \subset E(b_1) \times E(b_2) \times E(b_3)$, except for a finite number of values of $b_1, b_2, b_3 \in \CC$ (cf. Remark \ref{fixingpts}). 

The following script  moreover proves that the 161 groups $G$ belong  to exactly  16 $\De$-orbits. 

\begin{small}
\begin{verbatim}
F:=FiniteField(2); V6:=VectorSpace(F,6); 
V3:=VectorSpace(F,3); H3:=Hom(V6,V3);
U:={ V6![0,0,0,0,0,1], V6![0,0,0,1,0,0], V6![1,0,0,0,0,0], 
	V6![1,0,0,0,1,0], V6![0,1,0,0,0,0], V6![0,1,0,1,1,1], 
	V6![0,0,1,1,0,0], V6![0,0,1,0,1,1], V6![1,1,1,0,0,1], 
	V6![1,1,1,1,1,0], V6![0,0,0,0,1,0], V6![0,0,0,1,0,1], 
	V6![0,0,1,0,0,0],  V6![1,0,0,0,0,1], V6![0,0,1,0,1,0],
V6![1,0,0,1,0,0]};
S3:={Kernel(f): f in H3 | Dimension(Kernel(f)) eq 3}; 
M:=[**];  
for k in S3 do K:=Set(k); 
    if #(K meet U) eq 0 then Append(~M, k ); end if; 
end for;	
#M;
161
P:={1..#M}; Q:={}; // Delta-action
g1:=hom<V6->V6| V6![0,0,0,1,0,0],V6![0,1,0,0,0,0],V6![0,0,0,0,1,0],
 V6![1,0,0,0,0,0],V6![0,0,1,0,0,0], V6![0,0,0,0,0,1]>;
g2:=hom<V6->V6| V6![0,0,0,0,0,1],V6![0,1,0,0,0,0], V6![0,0,1,0,0,0],
 V6![0,0,0,1,0,0],V6![0,0,1,0,1,0], V6![1,0,0,0,0,0]>;
f1:=hom<V6->V6| V6![1,0,0,0,0,0],V6![1,1,1,1,1,0], V6![0,0,1,0,0,0],
 V6![0,0,0,1,0,0],V6![0,0,0,0,1,0], V6![0,0,0,0,0,1]>;
f2:=hom<V6->V6| V6![1,0,0,0,0,0],V6![1,1,1,0,0,1], V6![0,0,1,0,0,0],
 V6![0,0,0,1,0,0],V6![0,0,0,0,1,0], V6![0,0,0,0,0,1]>;
f3:=hom<V6->V6| V6![1,0,0,0,0,0],V6![0,1,0,1,1,1], V6![0,0,1,0,0,0],
 V6![0,0,0,1,0,0],V6![0,0,0,0,1,0], V6![0,0,0,0,0,1]>;
L1:=Transpose(Matrix([g1(x): x in Basis(V6)]));
L2:=Transpose(Matrix([g2(x): x in Basis(V6)]));
H1:=Transpose(Matrix([f1(x): x in Basis(V6)]));
H2:=Transpose(Matrix([f2(x): x in Basis(V6)]));
H3:=Transpose(Matrix([f3(x): x in Basis(V6)]));
GL6:=GeneralLinearGroup(6,F); PG:=sub<GL6|L1,L2,H1,H2,H3>; 
while not IsEmpty(P) do 
	i:=Rep( P ); Exclude(~P,i);  Include(~Q,i);
	for m in PG do 
  f:=map<V6->V6| x:->[(m[1],x),(m[2],x),(m[3],x),
                           (m[4],x),(m[5],x),(m[6],x)]>;
  test:=sub<V6|f(Set(M[i]))>; 
  if  exists(x){x: x in P | M[x] eq test } then 
        Exclude(~P, x); end if;
end for; end while;
#Q;
16  
\end{verbatim} 
\end{small}
\vspace{-.4cm}

This proves the first assertion.

\noindent
2) In tables \ref{q0}-\ref{q3} we list one representative $G$ for each of the 16 $\De$-orbits. 

Observe that the dimension of the family is 3 (the number of moduli of the three elliptic curves $E(b_1) \times E(b_2) \times E(b_3)$) if and only if the group $G$ stabilizes only a finite number of Burniat hypersurfaces, equivalently iff $G$ is neither contained in $\mathcal{G}_1$ nor in $\mathcal{G}'_1$. Otherwise $G$ is contained in $\mathcal{G}_1$ or in $\mathcal{G}'_1$ and, by Lemma \ref{stabs}, fixes a pencil of Burniat hypersurfaces. Therefore in this case the dimension of the family of generalized Burniat type surfaces is 4.

\noindent
It is now easy to verify that in case $\mathcal S_1$ (of Table \ref{q0}) $G \subset \mathcal{G}'_1$, in case 
$\mathcal S_2$ $G \subset \mathcal{G}_1$, whereas in cases $\mathcal S_3$-$\mathcal S_{16}$ of Tables \ref{q0}-\ref{q3} $G$ is not contained in any of the two groups $\mathcal{G}_1, \mathcal{G}'_1$. 
\end{proof}

\subsection{The fundamental groups}
To determine the fundamental group of a GBT surface $S\rightarrow X =\hat X/G$, 
we preliminarily observe that, by van Kampen's theorem and since   $X$ has only nodes as singularities,
$ \pi_1(X)= \pi_1(S)$. Then  we argue as follows.

Let  $E_j=\mathbb{ C}/ \langle e_j, \tau_j e_j\rangle$, $j=1,2,3$ and 
denote by $\Lambda$ the fundamental group of $A^0:=E_1\times E_2 \times E_3$. In particular, 
$\Lambda= \Lambda_1\oplus \Lambda_2\oplus \Lambda_3$,
where $\Lambda_j= \langle e_j, \tau_j e_j\rangle$.

\begin{lemma}
Consider the affine group
$$\Gamma:=\langle\gamma_1,\gamma_2,\gamma_3, e_1, \tau_1 e_1, e_2,\tau_2 e_2, e_3,\tau_3 e_3 
\rangle\leq \mathbb{A}(3,\mathbb{C})\,,$$ generated by $\Lambda$ and by 
 lifts  $\gamma_j$  of the generators $g_j$ of $G$ as affine transformations.
 
 Then $\Gamma=\pi_1(X)= \pi_1(S)$.
\end{lemma}

\begin{proof}

Observe that by the  Lefschetz' hyperplane section theorem (see \cite[Theorem 7.4]{milnor_morse})
 follows  that
$\pi_1(\hat X)\cong \pi_1(A^0)= \Lambda$. 

\noindent
The universal covering $\tilde{X}$ of $\hat X\subset A^0$ has then a natural 
inclusion $\tilde{X}\subset \mathbb C^3$ and the affine group $\Ga$ 
 acts on $\mathbb{C}^3$
 leaving $\tilde{X}$ invariant. Since the action of $\Ga$ on $\tilde{X}$  is free, and 
$X= \hat X/G= \tilde X/\Gamma$ we conclude that  $\Gamma=\pi_1(X)= \pi_1(S)$.
\end{proof}

The following  MAGMA script, which is an extended version of  the previous one, computes the fundamental group of each GBT surface. 
Observe that the fundamental group does only depend on $G$: since  it does not change within the same connected   family, and since each group $G$
determines an irreducible family.

\begin{small}
\begin{verbatim}
V9:=VectorSpace(F,9); T:=[* *];
h:=hom<V6->V9| V9![1,0,0,0,0,0,0,0,0], V9![0,1,0,0,1,0,0,1,0],
 V9![0,0,1,0,0,0,0,0,1], V9![0,0,0,1,0,0,0,0,0], 
 V9![0,0,0,0,0,1,0,0,1], V9![0,0,0,0,0,0,1,0,0]>;
G1:=DirectProduct([CyclicGroup(2),CyclicGroup(2),CyclicGroup(2)]);
G2:=DirectProduct([CyclicGroup(2),CyclicGroup(2),CyclicGroup(2)]);
G3:=DirectProduct([CyclicGroup(2),CyclicGroup(2),CyclicGroup(2)]);
H:=DirectProduct([G1,G2,G3]);
PolyGroup:=func<seq|Group<a1,a2,a3,a4| 
           a1^seq[1], a2^seq[2],a3^seq[3],a4^seq[4], a1*a2*a3*a4>>;
P1:=PolyGroup([2,2,2,2]);
P2:=PolyGroup([2,2,2,2]);
P3:=PolyGroup([2,2,2,2]); 
P:=DirectProduct([P1,P2,P3]);
f:=hom<P->H | H!(1,2),H!(3,4),H!(5,6),H!(1,2)(3,4)(5,6),
H!(7,8),H!(9,10),H!(11,12),H!(7,8)(9,10)(11,12),
H!(13,14),H!(15,16),H!(17,18),H!(13,14)(15,16)(17,18)>;
for i in Q do  G:=h(M[i]);
  s:=[]; for j in [1..3] do s[j]:=Id(H); end for;
  for i  in {1..3} do
    for j in [1..9]  do
      if (G.i)[j] eq 1 then s[i]:=s[i]* H!(2*j-1,2*j);  
  end if; end for; end for;
  GG1:=sub<H|s>;  
  Pi1:=Simplify(Rewrite(P,GG1@@f)); 
  Append(~T, [* G, Pi1, AbelianQuotient(Pi1) *]); 
end for;
\end{verbatim}
\end{small}
Since the fundamental groups are infinite and the presentations given as output are quite long, we only list the respective first homology groups
for the $16$ families of surfaces in Tables \ref{q0}, \ref{q1}, \ref{q2} and \ref{q3}.

\noindent It is not obvious, from the presentation given as output of the MAGMA script,
 whether two of these fundamental groups are isomorphic or not. 
To check whether two different  families have different fundamental groups,
 we can compare the number of normal subgroups of the fundamental group of index $k \leq m$ (in our case $m=6$).
This can be done easily using the MAGMA function: \verb+LowIndexNormalSubgroups(H, m)+
which returns a sequence  containing the normal subgroups of the finitely presented group $H$ of index $k \leq m$.
This allows us to see that the fundamental groups of the families we constructed are 
pairwise non-isomorphic, except for two pairs: 
$(\mathcal S_{11}, \mathcal S_{12})$ and $(\mathcal S_{14}, \mathcal S_{15})$.\\
Indeed in these cases, the fundamental groups are isomorphic.
 We verified this using the MAGMA function \verb+SearchForIsomorphism(H, K, n)+
which attempts to find an isomorphism of the finitely presented group $H$ with the finitely presented group $K$. 
The search is restricted to those homomorphisms for which the sum of the word-lengths of the images 
of the generators of $H$ in $K$ is at most $n$ (in our case $n=8$). The answer is given as follows:
if an isomorphism $\phi$ is found, then the output is the triple (true, $\phi$, $\phi^{-1}$); otherwise, the output is `false'. 

That these isomorphisms exist is no coincidence: we shall in fact show later that in both cases we have two families
of   surfaces which are contained in a larger   irreducible family (see Sections \ref{modspace} and \ref{BdFthree}).

Since for a smooth projective surface $S$ it holds $q(S)=\frac12 \rk H_1(S, \mathbb Z)$, we have proved the following: 

\begin{theorem}\label{fundgroup}
Among the 16 families of generalized Burniat type surfaces four have $p_g=q=0$ (Table \ref{q0}), 
eight have $p_g=q=1$ (Table \ref{q1}), three have $p_g=q=2$ (Table \ref{q2}) and one has $p_g=q=3$ (Table \ref{q3}).\\
Moreover, the fundamental groups of these families are pairwise non isomorphic, except 
for $\pi_1(S_{11})\cong\pi_1(S_{12})$
and $\pi_1(S_{14})\cong\pi_1(S_{15})$, where $S_j$ is in the family $\mathcal S_j$.
\end{theorem}

\begin{remark}\label{eqGBT}
We observe that the family $S_2$ in Table \ref{q0} corresponds to  the family of {\it primary  Burniat surfaces} (cf. Section \ref{Inouedescription}).
\end{remark}


\section{The moduli space of generalized Burniat type surfaces}\label{modspace}

The aim of this section is to describe the connected components of the Gieseker moduli space of surfaces of general type containing
 the isomorphism classes of  the generalized Burniat type surfaces.

\noindent First we shall prove the following result:
\begin{theorem}
Let $X$ be the canonical  model of a generalized Burniat type surface $S$.
Then  the base of the Kuranishi family of $X$ is smooth.
\end{theorem}

\begin{proof}
Recall that $X$ is the quotient model of a generalized Burniat type surface $S$, and let $\hat{X} \rightarrow X$ be the canonical $G \cong (\mathbb Z / 2 \mathbb Z)^3$-cover. Then $\hat{X} \subset A^0=E_1 \times E_2 \times E_3$ is a hypersurface of multidegree $(2,2,2)$ having at most nodes as singularities.

\noindent
It suffices to show (cf. \cite[Proposition 4.5]{superficial})  that the base of the Kuranishi family of $\hat{X}$ is smooth
(since the base of Kuranishi family of $X$ is given by the $G$-invariant part of the base of the Kuranishi family of $\hat{X}$).

Now, since $\hat X$ moves in a family with smooth base of  dimension $13=6+7$, it it is enough to show that
$$\dim \Ext^1_{\mathcal{O}_{\hat X}} (\Omega^1_{\hat{X}},\mathcal{O}_{\hat X}) = 13. $$
Moreover, the Kodaira-Spencer map of the above family is a bijection, but we omit the verification here.

Indeed $\hat X \subset A^0$ is an ample divisor, and it suffices to apply the following lemma.
\end{proof}

\begin{lemma}\label{divisorinppav}
Let $A$ be an Abelian variety of dimension $n$ and let $D \subset A$ be an ample divisor. Then:
$$
\dim \Ext^1_{\mathcal{O}_D} (\Omega^1_D,\mathcal{O}_D) = \frac 12 n(n+1) + \dim |D|.
$$
\end{lemma}
\begin{proof}
Consider the exact sequence
$$
0 \rightarrow \mathcal O_D (- D) \rightarrow \Omega_{A}^1 \otimes \mathcal O_D \cong \mathcal O_D^{\oplus n} \rightarrow \Omega^1_D \rightarrow 0.
$$
Applying the functor $\Hom_{\mathcal O_D}(_-,\mathcal O_D)$, we obtain the long exact sequence:
\begin{equation}\label{ext}
\begin{array}{rl}
0 &\rightarrow \Hom(\Omega^1_{D},\mathcal O_{D})=0 \rightarrow \Hom(\mathcal O_{D}^{\oplus n},\mathcal O_{D}) \rightarrow \Hom(\mathcal O_{D} (-D),\mathcal O_{D}) \rightarrow  \\
&\rightarrow \Ext^1(\Omega^1_{D},\mathcal O_{D}) \rightarrow \Ext^1(\mathcal O_{D}^{\oplus n},\mathcal O_{ D}) \rightarrow \Ext^1(\mathcal O_{D} (-D),\mathcal O_{D}) \rightarrow \\
&\rightarrow \Ext^2(\Omega^1_{D},\mathcal O_{D}) \rightarrow \Ext^2(\mathcal O_{D}^{\oplus n},\mathcal O_{D}) \rightarrow \Ext^2(\mathcal O_{ D} (- D),\mathcal O_{D}) \rightarrow \ldots .
\end{array}
\end{equation}

\noindent We have that
\begin{itemize}
\item[i)] $\Ext^i(\mathcal O_{D}^{\oplus n},\mathcal O_{D}) = H^i(D, \mathcal O_{D})^{\oplus n}$;
\item[ii)] $\omega_{D} = \omega_{A} \otimes \mathcal O_{D} (D) = \mathcal O_{D} (D)$;
\item[iii)] $\Ext^i(\mathcal O_{ D} (- D),\mathcal O_{D}) \cong \Ext^i(\mathcal O_{ D},\mathcal O_{D}(D)) = \Ext^i(\mathcal O_{ D},\omega_{D}) \cong$ \\
$H^{n-1-i}(D,\mathcal O_{ D})^*$, where the last equality holds by Serre duality.
\end{itemize}
Next, we consider the short exact sequence:

$$
0 \rightarrow \mathcal O_A (- D) \rightarrow \mathcal O_A \rightarrow \mathcal O_{ D} \rightarrow 0
$$
and the associated long cohomology sequence

\begin{multline}\label{coh}
0 \rightarrow H^0(\mathcal O_A) \rightarrow H^0(\mathcal O_{ D}) \rightarrow H^1(\mathcal O_A (- D) )\rightarrow H^1(\mathcal O_A) \rightarrow \\\rightarrow H^1(\mathcal O_{ D}) 
\rightarrow H^2(\mathcal O_A (- D) )\rightarrow H^2(\mathcal O_A) \rightarrow \ldots 
\rightarrow H^{n-1}(\mathcal O_{ D}) \rightarrow \\
\rightarrow H^n(\mathcal O_A (- D) )\rightarrow H^n(\mathcal O_A) \rightarrow 0.
\end{multline}
Note that by Serre duality
$H^i(\mathcal O_A (- D) ) \cong H^{n-i}(\mathcal O_A (D) )^*$, and since $D \subset A$ is an ample divisor, we get that these cohomology groups are trivial for $i \leq n-1$ by the Kodaira vanishing theorem. 

\noindent This implies that
\begin{itemize}
\item $\dim H^i(\mathcal{O}_D) = \dim  H^i(\mathcal{O}_A) = \binom {n} {i}$, for $0 \leq i \leq n-2$,
\item $\dim H^{n-1} (\mathcal{O}_D) = \dim |D| + n$.
\end{itemize}

\noindent
Inserting this information in the long exact sequence (\ref{ext}), we see that 
$$\dim \Ext^1_{\mathcal{O}_D} (\Omega^1_{D},\mathcal{O}_D) = \frac 12 n(n+1) + \dim |D| \,,$$
once we show that
$$
\varphi \colon \Ext^1(\mathcal O_{D}^{\oplus n},\mathcal O_{ D}) \rightarrow \Ext^1(\mathcal O_{D} (-D),\mathcal O_{D}) 
$$
is surjective.

\noindent
But 
$$\Ext^1(\mathcal O_{D}^{\oplus n},\mathcal O_{ D}) \cong H^1(\mathcal O_{ D})^{\oplus n} \cong H^1(\mathcal O_A^{\oplus n}) \cong H^1(\Theta_A)
$$ and 
$$ \Ext^1(\mathcal O_{D} (-D),\mathcal O_{D}) \cong H^{n-1-1}(\mathcal O_{ D})^* \cong H^{n-2}(\mathcal O_A)^* \cong H^2(\mathcal O_A),
$$ where the first and third equality follow from Serre duality.

\noindent
Composing with these isomorphisms, $\varphi$ becomes
$$
H^1(\Theta_A) \rightarrow H^2(\mathcal O_A),
$$
the contraction with the first Chern class of $D$, an element of $H^1(A,\Omega^1_A)$, which is represented by a non degenerate alternating form. Hence the surjectivity of $\varphi$ follows.
\end{proof}

\subsection{Surfaces in $\mathcal S_j$ with  $j \leq 10$, i.e., with $p_g=q\leq 1$.}\quad 

Recall the following definition.

\begin{definition}[{\cite[Definition 0.2-0.3]{BC12}}] \label{DCIT}
A complex projective manifold $X$ is said to be a \textit{diagonal classical Inoue-type manifold} if
\begin{enumerate}	
\item $\dim(X)\geq 2$;
\item there is a finite group $G$ and a Galois \'etale $G$-covering $\hat X\rightarrow X\,(=\hat X/G)$ such that: 
\item $\hat X$ is an ample divisor inside a $K(\Gamma,1)$-projective manifold $Z$ (hence by Lefschetz 
$\pi_1(\hat X) \cong \pi_1(Z)\cong \Gamma$) and moreover
\item the action of $G$ on $\hat X$ yields a faithful action on $\pi_1(\hat X) \cong \Gamma$: in other words the exact sequence 
$$1 \rightarrow \Gamma\cong \pi_1(\hat X) \rightarrow \pi_1(X) \rightarrow G \rightarrow 1$$
gives an injection $G\rightarrow \mathrm{Out}(\Gamma)$, defined by conjugation;
\item $Z=(A_1\times \ldots \times A_r)\times  (C_1 \times \ldots \times C_s)$
where each $A_j$ is an Abelian variety and each $C_k$ is a curve of genus $g(C_k)\geq 2$;
\item the action of $G$ on $\hat X$ is induced by a diagonal action on $Z$;
\item the faithful action on $\pi_1(\hat X) \cong \Gamma$, induced by conjugation by lifts of elements of $G$, 
 has the Schur property:
 \begin{equation}\tag{SP}\label{SP}
 \mathrm{Hom}(V_j,V_k)^G=0 \,, \quad \forall k \neq j\,,
 \end{equation}
where $V_j:=\Lambda_j \otimes \mathbb Q$, being $\Lambda_j:=\pi_1(A_j)$ (it suffices  to verify that, for each $\Lambda_j$, there is a subgroup 
$H_j$ of $G$ for which  $\mathrm{Hom}(V_j,V_k)^{H_j}=0 \,, \forall k \neq j$).
\end{enumerate}

We say instead that $X$ is  a \textit{diagonal classical Inoue-type variety} if we replace the assumption of smoothness of $X$
by the assumption that $X$ has canonical singularities.
\end{definition}

To fix the notation, let us call a surface $S$ a {\it generalized Burniat type (GBT) surface of type $j$} if 
$S$ belongs to the  family $\mathcal S_j$ in Tables \ref{q0}-\ref{q3}.

\begin{lemma}
Let $X_j$ be the canonical model of a GBT surface $S_j$ of type $j$. Then the embedding $\hat X_j \subset A^0 = E_1 \times E_2 \times E_3$
realizes $X_j$ as a diagonal classical Inoue-type variety  if and only if $1\leq j\leq 10$.
\end{lemma}

\begin{proof}
It is trivial to see that the canonical model of a generalized Burniat type surface $X_j=\hat X/G_j$ satisfies conditions (1-6) in Definition \ref{DCIT}.
Hence there remains only  to  determine which surfaces  fulfill the Schur Property (\ref{SP}).\\
To verify the Schur Property one has to find, for each pair $j\neq k \in\{1,2,3\}$ an element $g\in G$ such that, $d g$ being the derivative of $g$, 
$\mathrm d g_{|E_j} \cdot \mathrm d g_{|E_k}= -1$.
Let $j=1$ and $g=(0,1,0,1,1,0,1,1,0)\in G_1$: then $\mathrm d g_{|E_1}= -1$,  $\mathrm d g_{|E_2}=\mathrm d g_{|E_3}= 1$,
while for $g'=(0,0,0,0,0,1,1,0,1)\in G_1$ one has $\mathrm d g'_{|E_2}=-1$ and $\mathrm d g'_{|E_3}= 1$.
Hence $X_1$ satisfies   (\ref{SP}). 
Considering a suitable pair of generators of $G_j$, one can prove in the same way that  $X_j$ satisfies  (\ref{SP}) for $j=2,\ldots, 10$.

Consider now the case $j=11$ and let $g$ be one of the three generators of $G_{11}$ in Table \ref{q1}.
 Then $\mathrm d g_{|E_1}= \mathrm d g_{|E_3}= -1$ and $\mathrm d g_{|E_2}=1$; this means that
 $\mathrm{Hom}(V_1,V_3)^G\neq0$, hence $X_{11}$ does not fulfill the Schur property.
 
\noindent In the same way one can show that  $S_j$ does not fulfill  (\ref{SP}) for $j=12,\ldots, 16$. 
 \end{proof}

We are now in the position to  prove the following result. 

\begin{theorem}
Let $S$ be a smooth projective surface homotopically equivalent to a GBT surface $S_j$ of type $j$ with $1\leq j \leq 10$. Then $S$ is a GBT surface of type $j$, i.e., contained in the same irreducible family as $S_j$.
\end{theorem}

\begin{proof}
Assume that $S$ is homotopically equivalent to $S_j$ ($1 \leq j\leq 10$), hence in particular $S$ has the same fundamental group as $S_j$. Consider the \'etale $G_j \cong (\mathbb Z / 2 \mathbb Z)^3$-cover $\hat S \rightarrow S$. Then by \cite[Theorem 0.5]{BC12} we have a splitting of the
Albanese variety and an Albanese map $f \colon \hat S \rightarrow E_1 \times E_2 \times E_3$ which is generically finite onto its image $W$. By loc.~cit.~Lemma 1.2, $G_j$ acts on $E_1 \times E_2 \times E_3$ with the same action as for a GBT surface of type $j$. It is now easy to verify that there is no effective $G_j$-invariant divisor of numerical type $(1,1,1)$ on $E_1 \times E_2 \times E_3$, hence $W$ has class $2F_1 + 2F_2 + 2F_3$, where $F_i$ is the class of a fibre of the projection of $E_1 \times E_2 \times E_3$ on the $j$-th factor. Therefore $f$ is birational onto its image
and one verifies as in loc.~cit.~that  $W$ has at most rational double points as singularities and is therefore the canonical model $\hat X$ of $\hat S$.

\noindent
{\bf Claim:} $W$ is   the pull-back of a Del Pezzo surface in $\mathbb P^1 \times  \mathbb P^1 \times \mathbb P^1$
for a suitable degree $(\ZZ/2 \ZZ)^3$-covering $\pi'$.

\noindent
{\it Proof of the claim.} The pull back of a divisor of multidegree $(1,1,1)$ on $\mathbb P^1 \times  \mathbb P^1 \times \mathbb P^1$ under any
$$
\pi' : E_1 \times E_2 \times E_3 \rightarrow \mathbb P^1 \times  \mathbb P^1 \times \mathbb P^1
$$
is a divisor which has the same class as
$W$: hence the two divisors are linearly equivalent to a translate of each other. Since   the corresponding linear systems have the same
dimension we infer that $W$ is the translate of such an effective divisor. Changing the origin of the Abelian variety $A^0$ we
obtain another action of $(\ZZ/2 \ZZ)^3$ such that $W$ is invariant; the claim is thus proven.

\noindent
We have therefore seen that $S$ is a GBT surface and has the same fundamental group as $S_j$. Thus by our classification $S$ is in the same irreducible family as $S_j$, whence $S$ is a 
GBT surface of type $j$.
\end{proof}

\begin{remark}
The same conclusion holds under the weaker assumptions:

1) $\pi_1(S) \cong \pi_1 ( S_j)$

2) the corresponding covering $\hat{S}$, whose Albanese is a product of $3$ elliptic curves because of the Schur property,
 satisfies that the image of the Albanese map has class $(2,2,2)$.

\end{remark}

We can now summarize our results in the following theorem
\begin{theorem}\label{moduli}
The connected component $\mathfrak N_j$ of the Gieseker moduli space $\mathfrak M_{1,6}^{can}$ corresponding to generalized Burniat type surfaces of type $j$ ($1 \leq j \leq 10$) is irreducible,  normal and unirational, of dimension 4 if $j=1$ or $2$, else   of dimension 3.

\end{theorem}

\begin{proof}
We have shown that the Kuranishi family is smooth, hence the moduli space is normal. By the previous theorem each family of GBT surfaces with $ j \leq 10$
surjects onto a connected component of the Gieseker moduli space: since the family has a rational base (a projective bundle over a rational
variety), follows the assertion about the unirationality.
\end{proof}

\subsection{Surfaces in $\mathcal S_j$ with $j=11,12$, having $p_g=q=1$.}\quad\\
Since these surfaces do not fulfill the Schur property, the family constructed as $(\mathbb{Z}/2 \mathbb{Z})^3$-quotient of a Burniat hypersurface in a product of three elliptic curves is not complete. We will study these surfaces in greater  generality in Section \ref{bdf} and Section \ref{BdFthree}. In fact, it turns out that the  families $\mathcal  S_{11}$, $\mathcal S_{12}$ yield two irreducible subsets each of codimension one  in an irreducible connected component of dimension 4 of the moduli space of surfaces of general type with $p_g=q=1$, $K^2 =6$.

\subsection{Surfaces in $\mathcal S_j$ with $j=13,14,15$, i.e., those with $p_g=q=2$} \quad \\
We have three families (each of dimension 3, the number of moduli of the triple of  elliptic curves) of GBT surfaces with $p_g = q =2$. 
We have already observed that  the embedding $ \hat X_j \subset A^0= E_1\times  E_2 \times  E_3$   does not fulfill the Schur property. In fact, it is not difficult to show that each of the three families is a subfamily of a four dimensional irreducibile family, where the product  of the two elliptic curves on which  the projection of $G_j$ acts freely deforms to an Abelian surface $A_2$.  In this case the embedding $ \hat X_j \subset  E_1\times  A_2$
fulfills the Schur property and we can show that we obtain in this way  exactly  two irreducible connected components of the moduli space of surfaces of general type.

\noindent
 We do not give more details here since  these surfaces have already been classified in \cite{penpol}. 
 
 Observe in fact the following:
 
 \begin{proposition}
 Let $S$ be a GBT surface with $p_g(S) = q(S) =2$. Then $S$ is of Albanese general type and the Albanese map is generically of degree 2.
 \end{proposition}

\begin{proof}
Assume $S$ to be of type $13$ (the proof in the other two cases is exactly the same) and consider the following diagram:
 
 $$
\xymatrix{ 
&\hat X\subset E_1\times E_2 \times E_3 \ar[d]_{G_{13}}\ar[r]^{\quad p_{23}}&E_2 \times E_3 \ar[d]^{p_{23}(G_{13})}\\
S\ar[r]&X \ar[r]^{a\qquad}& (E_2 \times E_3) /p_{23}(G_{13}) \ 
}$$

\noindent Since $p_{23} \colon \hat X \rightarrow E_2 \times E_3$ is generically finite of degree 2 (as $\hat X$
 is a divisor of multidegree $(2,2,2)$), and since $G_{13} \cong p_{23}(G_{13})$, one sees immediately that $a$ is generically finite of degree 2 and that 
$\Alb(S) = (E_2 \times E_3) /p_{23}(G_{13})$.
\end{proof}

We recall the following result due to Penegini and Polizzi:
\begin{theorem}[{\cite[Theorem 31]{penpol}}]\label{albdc}
Let $\mathcal{M}$ be the moduli space of minimal surfaces $S$ of general type with $p_g=q=2$, $K_S^2 =6$ and Albanese map of degree 2. Then the following holds:
\begin{itemize}
\item[(i)] $\mathcal{M}$ is the union of three irreducible connected components, namely $\mathcal{M}_{Ia}$, $\mathcal{M}_{Ib}$ and $\mathcal{M}_{II}$.
\item[(ii)] $\mathcal{M}_{Ia}$, $\mathcal{M}_{Ib}$ and $\mathcal{M}_{II}$ are generically smooth of respective dimensions 4, 4, 3.
\item[(iii)] The general surface in $\mathcal{M}_{Ia}$ and $\mathcal{M}_{Ib}$ has ample canonical class; all surfaces in $\mathcal{M}_{II}$ have ample canonical class.
\end{itemize}
\end{theorem}

It is immediately clear that the subset of the moduli space corresponding to GBT surfaces with $p_g=q=2$ cannot be contained in $\mathcal M_{II}$, 
 since it has dimension 3, while 
 the families $13,14,15$ yield  irreducible families of dimension at least four.
We have the following
\begin{lemma}\label{genus}
Let $S_j$ be a GBT surface of type $j$ with $p_g= q= 2$, i.e., $j \in \{13,14,15\}$. Consider the pencil $f_j \colon S_j \rightarrow \mathbb P^1 \cong E_k/p_k(G_j)$, where $k=1$ for $S_{13}$ and $k=3$ for $j=14,15$. Then the general fibre of $f_j$ is a smooth curve of genus 5 if $j=13$ and of genus 3 if $j=14,15$.
\end{lemma}

\begin{proof}
Consider the diagram
$$
\xymatrix{ 
&\hat X_j\subset E_1\times E_2 \times E_3 \ar[d]_{G_{j}\cong (\mathbb Z / 2)^3}\ar[r]^{\qquad p_k}&
E_k  \ar[d]^{p_k(G_{j})}\\
S_j\ar[r]&X_j \ar[r]^{f_j}& E_k /p_k(G_j) \ 
}$$

\noindent Note that the general fibre of $p_k$ is a divisor of bidegree $(2,2)$ in the product of two elliptic curves, whence has genus 5. Since $p_1(G_{13}) \cong (\mathbb Z/ 2 \mathbb Z)^3$, the genus of a general fibre of $f_{13}$ is 5, whereas $p_3(G_{14}), p_3(G_{15}) \cong (\mathbb Z/ 2 \mathbb Z)^2$, whence the genus of a general fibre of $f_{14}$ and $f_{15}$ is 3. 
\end{proof}

This allows us to conclude the following:
\begin{proposition}
Let $S_j$ be a GBT  surface of type $j$ with $p_g= q= 2$. Then the point of the Gieseker moduli space corresponding to $S_{13}$ lies in $\mathcal M_{Ia}$, whereas the point corresponding to $S_{14}$, resp.  $S_{15}$, lies in $\mathcal M_{Ib}$.

In particular, GBT surfaces with $p_g=q=2$ of type 13 (resp. 14, 15) form a three dimensional subset of the four dimensional connected component $\mathcal M_{Ia}$ (resp. $\mathcal M_{Ib}$).
\end{proposition}

\begin{proof}
This follows from Lemma \ref{genus} and  \cite[Proposition 22]{penpol}.
\end{proof}

\begin{remark}
Consider for $j=13$ (the same holds also for $j=14,15$) the irrational pencil $f \colon S_{13} \rightarrow E_2 / p_2(G_{13})$. 
Observe that $E_2 / p_2(G_{13})$ is an elliptic curve and that  the genus of the fibres of $f$ is $3$. This implies that $f$ is not isotrivial (otherwise it would be contained in the table of \cite{penegini}). This contradicts Theorem A of \cite{zucconi}.
\end{remark}

\subsection{Surfaces in $\mathcal S_{16}$, i.e., those with $p_g=q=3$.}\quad\\
Minimal surfaces of general type with $p_g=q=3$ are completely classified by the work of
several authors (cf. \cite{CCML98,Pir02,HP02}). 

\begin{theorem}
A minimal surface of general type with $p_g = q = 3$ has $K^2 = 6$ or $K^2 = 8$ and, more precisely:
\begin{itemize}
\item if $K^2 = 6$, $S$ is the minimal resolution of the symmetric square of a  curve of genus $3$;
\item otherwise $S = (C_2 \times C_3)/\sigma$, where $C_g$ denotes a curve of genus $g$ and $\sigma$ is
an involution of product type acting on $C_2$ as an elliptic involution (i.e.,
with elliptic quotient), and on $C_3$ as a fixed point free involution.
\end{itemize}
In particular, the moduli space of minimal surfaces of general type with
$p_g = q = 3$ is the disjoint union of $\mathcal M_{6,3,3}$ and $\mathcal M_{8,3,3}$,
 which are both irreducible of respective dimension 6 and 5.
\end{theorem}
We get:
\begin{proposition}
Generalized Burniat type surfaces with $p_g=q=3$  (i.e. of type 16) form a three dimensional subset of the six dimensional connected component  $\mathcal M_{6,3,3}$.
\end{proposition}

\section{Bagnera-de Franchis varieties}\label{bdf} 

\begin{definition}
A {\it Generalized Hyperelliptic Variety (GHV)} $X$ is  defined to be a quotient $X= A/G$ of an Abelian Variety $A$ by a nontrivial  finite group $G$ acting freely,
 and with the property  that $G$  contains no translations.
 \end{definition}
 Remark that, if $G$ is any group acting freely on $A$, which  is not a subgroup of the group of  translations, then 
 the quotient $X= A/G$ is a GHV. Because  the subgroup $G_T$ of translations in $G$ is a normal subgroup of $G$, and,
if we denote $G' = G/G_T$, then $A/G = A' / G'$, where $A'$ is the Abelian variety $ A' : = A/G_T$. 

\begin{definition}
1)  A {\it Bagnera-de Franchis  variety} (for short: BdF variety) is a GHV $ X= A/G $ such that $G \cong \mathbb{Z}/m \mathbb{Z}$ is a cyclic group.

\noindent
2) A  {\it Bagnera-de Franchis  variety  of product type} is a Bagnera-de Franchis  variety $ X= A/G $, where $A = (A_1 \times A_2)$,  $A_1, A_2$ are  Abelian Varieties, 
and $G \cong \mathbb{Z}/m \mathbb{Z}$ is generated by an automorphism of the form
$$ g(a_1, a_2 ) = ( a_1 + \beta_1, \alpha_2 (a_2)),$$

\noindent
where $\beta_1 \in A_1[m]$ is an element of order exactly $m$, and similarly $\alpha_2 : A_2 \rightarrow A_2$ is a linear automorphism
of order exactly $m$ without $1$ as eigenvalue (these conditions guarantee that the action is free). 

\noindent
3) If moreover all eigenvalues  of $\alpha_2 $
are primitive $m$-th roots of $1$, we shall say that  $ X= A/G $ is a {\it primary  Bagnera-de Franchis variety}.

\end{definition}

\begin{remark}
1) One can give a similar definition of Bagnera-de Franchis manifolds, requiring only that $A, A_1, A_2$ be complex tori.

2) If $A$ has  dimension $n=2$, the Bagnera-de Franchis manifolds coincide with the Generalized Hyperelliptic varieties, due
to the classification result of Bagnera-de Franchis in \cite{BdF}.
\end{remark}

We have the following proposition, giving a characterization of Bagnera-de Franchis varieties.

\begin{proposition}\label{quotprodtype}
Every Bagnera-de Franchis variety $ X= A/G$ is the quotient of a  Bagnera-de Franchis variety
of product type,  $(A_1 \times A_2)/ G$ by any finite subgroup $T$ of $(A_1 \times A_2)$ which satisfies the following properties:
\begin{enumerate}
\item $T$  is the graph of an isomorphism between two respective 
subgroups $ T_1 \subset A_1, \  T_2 \subset A_2$,
\item $(\alpha_2 - \Id) T_2 = 0$,
\item if $ g (a_1, a_2 ) = ( a_1 + \beta_1, \alpha_2 (a_2)) ,$ then the subgroup of order $m$ generated by $\beta_1$ intersects $T_1$ only in $\{0\}$.
\end{enumerate}
In particular, we may write $X$ as the quotient $X =  (A_1 \times A_2)/ (G \times T)$ of $ A_1 \times A_2 $ by the Abelian group $G \times T$. 
\end{proposition}

\begin{proof}
We refer to \cite{FabTop}. 
\end{proof}

\subsection{Actions of a finite group on an Abelian variety}

Assume that we have the  action of a finite group $G$ on  a complex torus 
 $ A = V / \Lambda$. Since every holomorphic map between complex tori lifts to a complex  affine map
 of the respective universal covers, we can attach  to the group $G$  the group of affine transformations
 $\Gamma$, which consists of all affine maps of $V$ which lift transformations of $G$. Then $\Gamma$ fits  into an exact sequence:
  $$1 \longrightarrow \Lambda \longrightarrow \Gamma \longrightarrow G \longrightarrow 1\,. $$
 
 \noindent
 The following is a slight improvement of \cite[Lemma 1.2]{BC12}:
 \begin{proposition}
 The  group $\Gamma$ determines the real affine type of the action of $\Gamma$ on $V$ (respectively: the rational affine type
 of the action of $\Gamma$ on $
 \Lambda \otimes \mathbb{Q}$), in particular the above
 exact sequence determines the action of $G$ up to real affine isomorphism of $A$ (resp.: rational affine isomorphism of $
( \Lambda \otimes \mathbb{Q} )/ \Lambda$).
 \end{proposition}
 
\begin{proof}
 It is clear that $ V = \Lambda \otimes_{\mathbb{Z}} \mathbb{R}$ as a real vector space, and we denote by $V_{\mathbb{Q}}: = 
 \Lambda \otimes \mathbb{Q}$.  
 Let $$\Lambda' : = \ker (  \alpha_L \colon \Gamma \rightarrow \GL  (V_{\mathbb{Q}}) \subset \GL  (V)) ,$$
$$\overline{G}_1 : = \im (  \alpha_L \colon \Gamma \rightarrow \GL  (V_{\mathbb{Q}}))\,.$$
The group $\Lambda'$ is obviously Abelian,  contains $\Lambda$, and maps isomorphically onto a lattice $\Lambda' \subset V$.

\noindent
In turn $ V = \Lambda'  \otimes_{\mathbb{Z}} \mathbb{R}$, and, if $ G' : = \Gamma / \Lambda'$, then $ G' \cong \overline{G}_1 $ and  the exact sequence 
$$1 \longrightarrow \Lambda' \longrightarrow \Gamma \longrightarrow G' \longrightarrow 1\,  ,$$
since we have an embedding $ G' \subset \GL(\Lambda')$, 
 shows that the affine group  $ \Gamma \subset \Aff (\Lambda') \subset \Aff (V) $ is uniquely determined
 ($\Gamma$ is the inverse image of $G'$ under $\Aff (\Lambda') \rightarrow \GL(\Lambda')$).

\noindent
There remains only to show that $\Lambda'$ is determined by $\Gamma$ as an abstract group, independently
of the exact sequence we started with. In fact, one property of $\Lambda'$ is that it is a maximal Abelian subgroup,
 normal and of finite index.

\noindent
Assume that $\Lambda''$ has the same property: then their intersection $\Lambda^0 : = \Lambda' \cap \Lambda''$ is a normal subgroup of
 finite index, in particular $\Lambda^0  \otimes_{\mathbb{Z}} \mathbb{R} = \Lambda'  \otimes_{\mathbb{Z}} \mathbb{R} = V $;
hence $\Lambda'' \subset \ker (  \alpha_L \colon \Gamma \rightarrow \GL  (V))= \Lambda'$,
where $\alpha_L$ is induced by conjugation on $\Lambda^0$.

\noindent
By maximality $\Lambda' = \Lambda''$.
\end{proof}

Observe that, in order to obtain the structure of a complex torus on $V / \Lambda'$, we must give a complex structure
on $V$ which makes the action of $G' \cong \overline{G}_1$ complex  linear. 

 In order to study the moduli spaces of the associated complex manifolds, we introduce  therefore a further invariant, called 
 the Hodge type, according to
the following
definition. 

\begin{definition}
 Given a  faithful representation $ G \rightarrow \Aut (\Lambda)$, where $\Lambda$ is a free Abelian group  of even rank $2n$,
 a {\it $G$-Hodge decomposition} is a $G$-invariant decomposition 
 $$\Lambda \otimes \mathbb{C} = H^{1,0} \oplus H^{0,1}, \ H^{0,1} = \overline{H^{1,0}}.$$ 
Write   $\Lambda \otimes \mathbb{C}$
as the sum of
isotypical components $$\Lambda \otimes \mathbb{C} = \oplus_{\chi \in \Irr (G)}
U_{ \chi}.$$

\noindent Write also  $U_{ \chi} = W_{ \chi} \otimes M_{ \chi}$, where $W_{ \chi} $ is the given irreducible representation,
and $ M_{ \chi} $ is a trivial representation of dimension $n_{ \chi} $.

\noindent
Then $V : = H^{1,0} =  \oplus_{\chi \in \Irr (G)}
V_{ \chi},$ where $V_{ \chi} = W_{ \chi} \otimes M^{1,0}_{ \chi}$ and $M^{1,0}_{ \chi}$ is a subspace of $M_{ \chi}$.
 The {\it Hodge type} of the decomposition
is the datum 
of the dimensions   $$\nu ( \chi): = \dim_{\mathbb{C}} M^{1,0}_{ \chi} $$
corresponding to  the Hodge summands for non real representations (observe in fact  that one must have: $\nu ( \chi) + \nu ( \bar{\chi}) = \dim ( M_{ \chi})$).
\end{definition}

\begin{remark}
Given a  faithful representation $ G \rightarrow \Aut (\Lambda)$, where $\Lambda$ is a free Abelian group  of even rank $2n$,
 all the  $G$-Hodge decompositions of a fixed Hodge type are parametrized by an open set in a product of Grassmannians.
 Since, for a non real irreducible representation $\chi$ one may simply choose $M^{1,0}_{ \chi}$ to be a complex subspace of dimension  $\nu ( \chi)$ of  $ M_{ \chi}$, and for $M_{ \chi} = \overline{(M_{ \chi})}$,
one simply chooses a complex subspace $M^{1,0}_{ \chi}$ of half dimension. Then the open condition is just that (since $ M^{0,1}_{ \chi} : = \overline {M^{1,0}_{ \chi} }$) we want $ M_{ \chi} = (M^{1,0}_{ \chi} ) \oplus  (M^{0,1}_{ \chi} ) $, or, equivalently,  $M_{ \chi} = (M^{1,0}_{ \chi}) \oplus  \overline{(M^{1,0}_{\bar{ \chi}} )}$.
\end{remark}

\subsection{Bagnera-de Franchis varieties of small dimension}

We have shown that a Bagnera-de Franchis  variety $X=A/G$ can be seen as the quotient of
 one of product type $(A_1 \times A_2)/G$ by a finite subgroup $T$ of $A_1 \times A_2$, satisfying the properties stated in Proposition \ref{quotprodtype}.
 
Dealing with appropriate choices of $T$ is the easy part, since,
as we saw,  the points $t_2$ of $T_2$ satisfy, by  property (2), $ \alpha_2 (t_2)  = t_2$.

 It suffices then to choose  $T_2 \subset  A_2[*]:= \ker (\alpha_2 - \Id_{A_2})$,
which is a finite  subgroup of $A_2$, and then  to pick an isomorphism $\psi \colon T_2 \rightarrow T_1 \subset A_1$,
such that $ T_1 : = \im (\psi) \cap \langle \langle \beta_1 \rangle \rangle = \{ 0\}$.

We therefore restrict ourselves from now on to  {\it Bagnera-de Franchis varieties of product type} and we show now how to further reduce to the investigation of {\it primary Bagnera-de Franchis varieties}.

In fact, in the case of a BdF variety of product type, $\Lambda_2$ is a $G$-module, hence a module over the group ring
$$ R : = R(m)  : = \mathbb{Z}[G]  \cong  \mathbb{Z}[x] / (x^m - 1).$$

The ring $R$ is in general far from being an integral domain, since indeed it can be written as a direct sum of cyclotomic rings, which 
are the integral domains
defined as  $R_k: = \mathbb{Z}[x] / (P_k(x))$. Here $P_k(x)$ is the $k$-th cyclotomic polynomial
$$ P_k(x) = \prod_{0 < j < k,\  (k,j)=1} (x - \epsilon^j)\,,$$ where $\epsilon = \exp ( 2 \pi i /k)$.
Then $$  R (m) = \oplus_ { k | m} R_k \,.$$ 

The following  elementary lemma, together with the splitting of the vector space $V$ as a direct sum of eigenspaces for $g$,
yields a decomposition of $A_2$ as a direct product $ A_2 =  \oplus_{ k | m} A_{2,k} $
of $G$ -invariant Abelian subvarieties $A_{2,k} $ on which $g$ acts with eigenvalues of order precisely $k$.

\begin{lemma}
Assume that $M$ is a module over a ring $R =  \oplus_k R_k$.
Then there is a unique direct sum decomposition 
$$M =  \oplus_k M_k,$$ 
such that 
\begin{itemize}
\item $M_k$ is an $R_k$-module, and 
\item the $R$-module structure of $M$ is obtained through the projections \ $ R \rightarrow R_k$.
\end{itemize}

\end{lemma}

\begin{proof}
We can write the identity in $R$ as a sum of idempotents $1 = \Sigma_k e_k$, where $e_k$ is the identity of $R_k$,
and $ e_k e_j = 0$ for $ j \neq k$.

\noindent
Then each element $w \in M$ can be written as
$$ w = 1 w = ( \Sigma_k e_k) w = \Sigma_k e_k w =  :  \Sigma_k w_k.$$ 

\noindent
Hence $ M_k$ is defined as $ e_k M$.
\end{proof}

\begin{remark}
1) If we have a primary Bagnera-de Franchis variety, then $\Lambda_2$ is a module over the integral domain $R : = R_m: = \mathbb{Z}[x] / (P_m(x))$. 

\noindent
Since $\Lambda_2$ is a projective $R$-module,  $\Lambda_2$  splits as the direct sum 
$\Lambda_2 = R^r \oplus I $ of a free module with an ideal
$ I \subset R$ (see \cite[Lemmas 1.5 and 1.6]{Milnor}), and $\Lambda_2$ is indeed free if the class number $h(R)=1$. The integers $m$ for  which this occurs are listed in the table on \cite[page 353]{washington}.

\noindent
2) To give a complex structure to $A_2 : = (\Lambda_2 \otimes_{\mathbb{Z}} \mathbb{R} )/ \Lambda_2$ 
it suffices to give a decomposition  $\Lambda_2 \otimes_{\mathbb{Z}} \mathbb{C} = V \oplus \overline{V}$,
such that the action of $x$ is holomorphic. This is equivalent to asking that $V$ is a direct sum of eigenspaces
$ V_{\lambda}$, for $\lambda =  \epsilon^j$ a primitive $m$-th root of unity. 

\noindent Writing  $U: = \Lambda_2 \otimes_{\mathbb{Z}} \mathbb{C} = \oplus U_{\lambda}$, the desired decomposition is obtained by choosing,
for each eigenvalue $\lambda$, a decomposition $  U_{\lambda} = U_{\lambda}^{1,0}  \oplus U_{\lambda}^{0,1} $ such that $ \overline{U_{\lambda}^{1,0}} = U^{0,1}_{\overline{\lambda}}$.

The simplest case (see \cite{cacicetraro} for more details) is the one where $I =0, r=1$, hence $  \dim (U_{\lambda} ) = 1$. Therefore we have only a finite number of complex structures,
depending on the choice of the $\frac{\varphi(m)}{2}$ indices $j$ such that $  U_{\epsilon^j} = U^{1,0}_{\epsilon^j}$ (here $\varphi(m)$ is the {\it Euler function}).
\end{remark}

\noindent
Observe that the classification of BdF varieties in small dimension is possible thanks to the observation that the $\mathbb{Z}$-rank of $R$ (or of any ideal $I \subset R$) cannot exceed 
the real dimension of $A_2$: in other words we have 
$$\varphi(m) \leq 2 (n-1),$$
where $\varphi(m)$ is the Euler function, which is multiplicative for relatively prime numbers,
and satisfies $\varphi(p^r) = (p-1)  p^{r-1}$, if $p$ is a prime number.

\noindent
For instance, if $n \leq 3$, then $\varphi(m) \leq 4$. Observe that  $\varphi(p^r) \leq 4$ iff 
\begin{itemize}
\item $p=3$, $5$ and $r=1$, or 
\item $p=2$, $r \leq 3$.
\end{itemize}
Hence, for $n \leq 3$, the only possibilities for $m$ are 
\begin{itemize}
\item $\varphi(m) =1$: $m=2$;
\item $\varphi(m) =2$: $m = 3,4,6$;
\item $\varphi(m) =4$: $m = 5, 8, 10, 12$.
\end{itemize} 
The classification is then also made easier by the fact that, in the above range for $m$, $R_m$ is a P.I.D., hence every  torsion free module is free. In particular $\Lambda_2$ is a free $R$-module.

\noindent
The classification for $n =  4$, since we must have $\varphi(m) \leq 6$, is going to include also the case $m=7, 9$.

We state now a result which will be  useful in Section \ref{BdFthree}.

\begin{proposition}
The Albanese variety of a Bagnera-de Franchis variety $X = A/G$ is the quotient $A_1 / (T_1 + \langle \langle \beta_1 \rangle \rangle)$.
\end{proposition}
\begin{proof}
Observe that the Albanese variety  $H^0(\Omega^1_X)^{\vee}/ \im (H_1(X, \mathbb{Z}))$ of $X = A/G$ is a quotient of the vector space $V_1$ by the image of the fundamental group of $X$
(actually of its abelianization, the  first homology group $H_1(X, \mathbb{Z})$):
since the dual of $V_1$ is the space of $G$-invariant forms on $A$, $H^0(\Omega^1_A)^G \cong H^0(\Omega^1_X)$.

We also observe that there is a well defined  map $X \rightarrow A_1 / (T_1 + \langle \langle \beta_1 \rangle \rangle)$, since $T_1$ is the first projection of $T$.
The image of the fundamental group of $X$ contains the image of $\Lambda$, which is precisely the extension of $\Lambda_1$
by the image of $T$, namely $T_1$. Since we have the exact sequence 
$$ 1 \longrightarrow \Lambda = \pi_1 (A) \longrightarrow   \pi_1 (X) \longrightarrow   G \longrightarrow 1$$
the image of the fundamental group of $X$ is generated by the image of $\Lambda$ and by the image of the transformation
$g$, which however acts on $A_1$ by translation by $\beta_1 = [b_1]$.
\end{proof}

\begin{remark}
Unlike the case of complex dimension $n=2$, there  are   Bagnera-de Franchis varieties $X = A/G$  with trivial canonical divisor,
for instance an elementary  example is  given by
 any  BdF variety which is standard (i.e., has  $m=2$) and is such that  $A_2$ has even dimension.

\end{remark}

 \subsection{Line bundles  on quotients and linearizations}
 
 Recall the following well known result (see Mumford's books \cite{abvar}, \cite{GIT}).
  
  \begin{proposition}\label{linearization}
  Let $ Y = X/G$ be a quotient algebraic variety and let  $p \colon X \rightarrow Y$ be the quotient map. Then:
  \begin{enumerate}
 \item there is a functor  between  
 \begin{itemize}
 \item line bundles $\mathcal{L}'$ on $Y$ and 
 \item $G$-linearized line bundles  $\mathcal{L}$, 
 \end{itemize}
 associating to $\mathcal{L}'$ its pull back $p^* (\mathcal{L}')$.
\item The functor   $\mathcal{L} \mapsto p_*(\mathcal{L})^G$ is a right inverse
  to the previous one, and $p_*(\mathcal{L})^G$ is invertible if the action is free, or if  $Y$ is smooth.
 \item Given a line bundle  $\mathcal{L}$ on $X$, it admits a $G$-linearization if and only if there is a Cartier divisor $D$ on $X$, which is 
  $G$-invariant and such that $\mathcal{L} \cong \mathcal{O}_X(D) = \{ f \in \mathbb{C}(X) | \divi(f) + D \geq 0\}.$ 
\item A necessary condition for the existence of a $G$-linearization on a line bundle $\mathcal{L}$ on $X$ is that
\begin{equation}\label{neccond}
 \forall g \in G, \ g^* ( \mathcal{L} ) \cong \mathcal{L}.
 \end{equation}
  \end{enumerate}
  If condition (\ref{neccond}) holds for $(\mathcal{L},G)$, one defines the {\it Theta group} of $\mathcal{L}$ as:
  $$ \Theta (\mathcal{L}, G) : =  \{ (\psi , g ) | g \in G, \ \psi : g^* ( \mathcal{L} ) \rightarrow \mathcal{L} \ {\rm is \  an \ isomorphism} \},$$
  and there is an exact sequence
\begin{equation}\label{theta}
 1 \longrightarrow \mathbb{C}^* \longrightarrow  \Theta (\mathcal{L}, G) \longrightarrow G \longrightarrow 1.
 \end{equation}
  \begin{itemize}
\item The splittings of the above sequence correspond to the    $G$-linearizations of $\mathcal{L}$.
 \item If the sequence splits, the  linearizations are a principal homogeneous space over the dual group $ \Hom (G, \mathbb{C}^*) = : G^*$ of $G$
(namely, each linearization  is obtained from a fixed one  by multiplying with an arbitrary element in $ \Hom (G, \mathbb{C}^*) = : G^*$).
\end{itemize}
  \end{proposition}

Thus, the question of the existence of a $G$-linearization on a line bundle $\mathcal{L}$ is reduced to the algebraic question of the splitting of the central extension (\ref{theta})
   given by the Theta group. This question is addressed by group cohomology theory, as follows (for details see \cite{BAII}).
   
      \begin{corollary}
   Let $\mathcal{L}$ be an invertible sheaf on $X$, whose class in $\Pic(X)$ is $G$-invariant. 
   Then there exists a $G$-linearization of $\mathcal{L}$ if and only if  the extension  class 
  $ [\psi] \in H^2(G, \mathbb{C}^*)$ of the exact sequence (\ref{theta}) induced by the Theta group $\Theta (G, \mathcal{L})$ is trivial.
  \end{corollary}
  
  \noindent
  The group $H^2(G, \mathbb{C}^*)$ is the group of {\it Schur multipliers} (see again \cite[page 369]{BAII}). 
  
  Schur multipliers occur naturally when we have a projective representation of a group $G$.
  Since, if we have a homomorphism $ \varphi \colon G \rightarrow \mathbb{P} \GL(r, \mathbb{C})$, we can pull back the central extension
  $$ 1 \longrightarrow \mathbb{C}^* \longrightarrow    \GL(r, \mathbb{C})  \longrightarrow  \mathbb{P} GL(r, \mathbb{C}) \longrightarrow 1$$
  via $\varphi$, we obtain an exact sequence
   $$ 1 \longrightarrow \mathbb{C}^* \longrightarrow    \hat{G}  \longrightarrow G\longrightarrow 1,$$
  and the extension class $[\psi] \in   H^2(G, \mathbb{C}^*)$ is the obstruction to lifting the projective representation
  to a linear representation $ G \rightarrow  \GL(r, \mathbb{C})$.
  
  \noindent
  It is an important remark that, if the group $G$ is finite, and $ n  = \mathrm{ord} (G)$, then the cocycles take values in the group 
  of roots of unity $\mu_n : = \{ z \in \mathbb{C}^* | z^n = 1 \}$.

   \begin{remark}
  1) Let $E$ be an elliptic curve with origin  $O$, and let $G$ be the group of $2$-torsion points $G : = E[2] \cong (\mathbb{Z}/2 \mathbb{Z})^2$, acting by translations
  on $E$. The divisor class of $2 O$ is never represented by a $G$-invariant divisor, since all  the $G$-orbits consist of $4$ points,
  and the degree of $2O$ is not divisible by $4$. Hence, $\mathcal{L} : = \mathcal{O}_E (2O)$ does not admit a $G$-linearization.
  However, we have a projective representation on $\mathbb{P}^1 = \mathbb{P} (H^0( \mathcal{O}_E (2O)))$, where each non zero element $\eta_1$ of the group
  fixes 2 divisors: the sum of the two points corresponding to $\pm \frac{\eta_1}{2}$, and its translate by another element $\eta_2 \in E[2]$.
  
  The two group generators  yield two linear transformations, which act on $ V : = H^0( \mathcal{O}_E (2O))= \mathbb{C} x_0 \oplus \mathbb{C} x_1$ as follows:
  $$ \eta_1(x_0) = x_1,  \eta_1(x_1) = x_0,  \ \eta_2 (x_j) = (-1)^j x_j.$$ 
  The linear group generated is however $D_4 \neq G$, since $$\eta_1\eta_2 (x_0) = x_1, \ \eta_1\eta_2 (x_1) = - x_0.$$
 2) 
The previous example is indeed a special case of  the {\it Heisenberg extension}, and $V $ generalizes to the
 {\it Stone-von Neumann  representation} associated   to an Abelian group $G$.
 
 \noindent This is simply  the space $ V : = L^2 (G, \mathbb{C})$ of square integrable functions on $G$ (see \cite{igusa},\cite{abvar}):
 
\begin{itemize}
\item $G$ acts on $ V : = L^2 (G, \mathbb{C})$ by translation $f (x) \mapsto f (x - g)$, 
 \item $G^*$ acts on $V$ by multiplication with the given character  $ f (x) \mapsto f (x) \cdot \chi (x)$, and 
 \item the commutator $[g, \chi ]$ acts on $V$ by the scalar multiplication with the constant $ \chi (g)$.
 \end{itemize}
 The Heisenberg group is the group of automorphisms of $V$ generated by $G$, $G^*$ and by $\mathbb{C}^*$ acting by  scalar multiplication.
Then  there is a central extension
 $$ 1 \longrightarrow \mathbb{C}^* \longrightarrow \Heis(G) \longrightarrow G \times G^* \longrightarrow 1 ,$$
 whose class in $H^2 ( G \times G^*, \mathbb{C}^*)$ is given by the $\mathbb{C}^*$-valued  bilinear form  $$\beta \colon (g, \chi) \mapsto  \chi (g) \in \Lambda^2 ( \Hom ( G \times G^*, \mathbb{C}^*)) \subset  H^2 ( G \times G^*, \mathbb{C}^*).$$
 
 The relation with Abelian varieties $ A = V / \Lambda$ is through the Theta group associated to an ample divisor $L$.
 
 \noindent
 In fact,  by the theorem of Frobenius the alternating form $ c_1(L) \in H^2(A, \mathbb{Z}) \cong \wedge^2 ( \Hom (\Lambda, \mathbb{Z}))$ admits,
 in a suitable basis of $\Lambda$, the normal form
 \begin{equation}
 D: =   \begin{pmatrix}
    0  &    D' \\
  - D' &    0
  \end{pmatrix} ,
  \end{equation}
  where $D' : = \diag (d_1, d_2, \dots , d_g)$, $d_1\mid d_2 \mid \dots \mid d_g$.

If one sets $G : = \mathbb{Z}^g/ D' \mathbb{Z}^g$, then $L$ is invariant under $G \times G^* \cong G \times G  \subset A$, acting by translation, and  the Theta group of $L$  is just 
isomorphic to the Heisenberg group $\Heis(G)$. 

  \end{remark}
  
  The nice part of the story is the following  very useful result, which was used by Atiyah in the case of elliptic curves  to study vector bundles on these (cf. \cite{atiyah}). We give a proof even if the result is well known.
  
  \begin{proposition}\label{Heisenberg}
  Let $G$ be a finite Abelian group, and let $ V : = L^2 (G, \mathbb{C})$ be the Stone-von Neumann representation.
  Then $V \otimes V^{\vee}$ is a representation of $G \times G^*$ and splits as the direct sum of all the characters of  $G \times G^*$.
  \end{proposition}
\begin{proof}
  Since the centre $\mathbb{C}^*$ of the Heisenberg group $ \Heis(G)$ acts trivially on $V \otimes V^{\vee}$, we have that $V \otimes V^{\vee}$ is a representation of $G \times G^*$.
  Observe that $G \times G^*$ is equal to its  group of characters, and its cardinality equals  the dimension of  $V \otimes V^{\vee}$,
  hence it suffices (and it will also be useful  for applications) to write for each character of $G \times G^*$ an explicit eigenvector.
  
  \noindent
  We shall use the letters $g, h, k$ for elements of $G$, and the greek letters $\chi, \eta, \xi$ for elements in the dual group.
  Observe that $V$ has two bases, one given by $\{ g \in G\}$, and the other given by the characters $\{ \chi \in G^* \}$.
 The {\it Fourier transform} $\mathcal{F}$ yields an isomorphism of the vector spaces $ V : = L^2 (G, \mathbb{C})$ and  $ W : = L^2 (G^*, \mathbb{C})$:
  $$ \mathcal{F} (f) := \hat{f}, \  \hat{f}(\chi) : = \int  f(g) ( \chi, g) \ \mathrm{d}g.$$
  The action of $h \in G$ on $V$ sends $ f(g) \mapsto f (g-h)$, hence for the characteristic functions in $\mathbb{C}[G]$, $h\in G$ acts as $ g \mapsto g + h$.
Instead   $\eta \in G^*$ sends $ f \mapsto  f \cdot \eta$, hence $ \chi \mapsto \chi + \eta$. Note that we use the additive notation also for the group of characters. 
  
Restricting $V$ to the {\it finite Heisenberg group}, which is a central extension of $ G \times G^*$ by $\mu_n$, we get a unitary representation,
hence we identify $V^{\vee}$  with $\bar{V}$. Then a basis of $V \otimes \bar{V}$ is given by the set $\{ g \otimes \bar{\chi}\}$.

\noindent
Given a vector $w:= \sum_{g, \chi} a_{g, \chi}  (g \otimes \bar{\chi} ) \in V \otimes \bar{V}$, then the action by $h \in G$ is given by
$$ h(w) = \sum_{g, \chi}  (\chi, h) a_{g -h, \chi}  (g \otimes \bar{\chi} ),$$ while the action by $ \eta \in G^*$ is given by 
$$\eta(w) =  \sum_{g, \chi}  (\eta, g) a_{g, \chi  - \eta}  (g \otimes \bar{\chi} ).$$

\noindent
Hence one verifies right away that 
$$  F_{k, \xi} : =  \sum_{g, \chi}  ( \chi - \xi, g - k)  (g \otimes \bar{\chi} ) $$ 
is an eigenvector with character $ (\xi, h) (\eta, k)$ for $ (h, \eta) \in ( G \times G^*)$.
   \end{proof}

\section{A surface in a Bagnera-de Franchis threefold} \label{BdFthree}
\noindent
Let $A_1$ be an elliptic curve, and let $A_2$ be an Abelian surface together with a line bundle $L_2$ yielding a polarization of type $(1,2)$. 

 \noindent  Take on $A_1$ the line bundle $L_1=\mathcal{O}_{A_1} ( 2 O)$,
and let $L$ be the line bundle on $A' : = A_1 \times A_2$, obtained as the exterior tensor product of $L_1$ and $L_2$,
so that
$$ H^0 (A', L) =  H^0 (A_1, L_1)  \otimes  H^0 (A_2, L_2) .$$
Moreover, we choose the origin in $A_2$ such that the space of sections $H^0 (A_2, L_2) $ consists only of {\it even} sections
(hence, we shall no longer be free to further change the origin by an arbitrary  translation).

\noindent
We want to construct a Bagnera-de Franchis variety $X: = A/ G$, where 
\begin{itemize}
\item $A = (A_1 \times A_2) / T$, and $G \cong T \cong \mathbb{Z}/2 \mathbb{Z}$, such that
\item there is a $G\times T$ invariant divisor $D \in |L|$, whence we get a surface $S =D/(T \times G) \subset X$, with $K_S^2 =  \frac{1}{4} K_D^2 =
\frac{1}{4} D^3 = 6$.
\end{itemize}
Write as usual $A_1 = \mathbb{C} / \mathbb{Z} \oplus \mathbb{Z} \tau$, and let $A_2 = \mathbb{C}^2 / \Lambda_2$. Suppose moreover, that $\lambda_1,\lambda_2, \lambda_3, \lambda_4$ is
a basis of $\Lambda_2$ such that with respect to this basis  the Chern class of $L_2$ is in Frobenius normal form. Let then $G=\langle g \rangle \cong \mathbb{Z}/ 2 \mathbb{Z}$ act on $A_1 \times A_2$ by
\begin{equation}\label{BCF}  g(a_1, a_2 ) : = (a_1 + \frac{\tau}{2}, - a_2 + \frac{\lambda_2}{2}),
\end{equation}
and define  $T : = ( \mathbb{Z}/2 \mathbb{Z}) ( \frac 12 ,\frac{\lambda_4}{2})$.

\noindent
Now, $G \times T$ surjects onto the group of two torsion points $A_1[2]$ of the elliptic curve,
and also on the subgroup $ ( \mathbb{Z}/2 \mathbb{Z})  ( \lambda_2 / 2) \oplus ( \mathbb{Z}/2 )  ( \lambda_4 / 2) \subset A_2[2]$.
Moreover,  both $H^0 (A_1, L_1) $ and $  H^0 (A_2, L_2)$ are the Stone-von Neumann representation
of the finite Heisenberg group of $G$, which is a central $\mathbb{Z}/2\mathbb{Z}$ extension of $ G \times T$.

By Proposition \ref{Heisenberg}, since in this case $ V \cong \overline{V}$ (the only roots of unity occurring are just $\pm 1$),
we conclude that there are exactly 4 divisors in $|L|$, invariant by:
\begin{itemize}
\item $ (a_1, a_2) \mapsto (a_1, - a_2)$ (since the sections of $L_2$ are even), 
\item $(a_1, a_2)  \mapsto  (a_1 + \frac{\tau}{2}, a_2 + \frac{\lambda_2}{2})$, and  
\item $(a_1, a_2)  \mapsto  (a_1 + \frac 12, a_2 + \frac{\lambda_4}{2})$.
\end{itemize}

Hence these four divisors descend to give four surfaces $S_i \subset X$, $i\in \{1,2,3,4\}$.

\begin{theorem}\label{BCF2}
Let $S$ be a minimal surface of general type with invariants $K_S^2 = 6$, $p_g(S) =q(S) = 1$ such that 
\begin{itemize}
\item there exists an unramified double cover
$ \hat{S} \rightarrow S$ with $ q (\hat{S}) = 3$, and such that 
\item the Albanese morphism $ \hat{\alpha} \colon  \hat{S}  \rightarrow A = \Alb(\hat{S})$ is birational onto its image $Z$,
a divisor in $A$ with  $ Z^3 = 12$.
\end{itemize}

\noindent
1) Then the canonical model of $\hat{S}$ is isomorphic to $Z$, and the canonical model of $S$ is isomorphic to $Y = Z / (\mathbb{Z}/2 \mathbb{Z})$, 
which a divisor in a Bagnera-de Franchis threefold $ X: = A/ G$, where $A = (A_1 \times A_2) / T$, $ G \cong T \cong \mathbb{Z}/2 \mathbb{Z}$,
and where the action is as in (\ref{BCF}).

\noindent
2) These surfaces exist, have an irreducible four dimensional moduli space, and their Albanese map $\alpha \colon S \rightarrow A_1 = A_1/ A_1[2]$ has 
general fibre a non hyperelliptic curve of genus $g=3$.

\end{theorem}

\begin{proof}
By assumption the Albanese map  $ \hat{\alpha} \colon  \hat{S}  \rightarrow A$ is birational onto $Z$, and we have $ K_{ \hat{S}}^2 = 12 = K_Z^2$,
since $\mathcal{O}_Z(Z)$ is the dualizing sheaf of $Z$.

\noindent
We shall argue similarly to \cite[Step 4 of Theorem 0.5, page 31]{BC12}. Denote by $W$ the canonical model of $\hat{S}$, and observe that by adjunction (see loc. cit.)
we have $ K_W =  \hat{\alpha}^* (K_Z ) - \mathfrak A$, where $\mathfrak A$ is an effective $\mathbb{Q}$-Cartier divisor.

\noindent
We observe now that $K_Z$ and $K_W$ are ample, hence we have an inequality,
$$ 12 = K_W^2 =  (\hat{\alpha}^* (K_Z ) - \mathfrak A)^2 = K_Z^2 - (\hat{\alpha}^* (K_Z ) \cdot  \mathfrak A) - (K_W \cdot  \mathfrak A) \geq K_Z^2 = 12,$$
and since both terms are equal to $12$, we conclude that $\mathfrak A= 0$, which means that $K_Z$ pulls back to $K_W$, whence
$W$ is isomorphic to $Z$. We have  a covering involution $ \iota \colon \hat{S} \rightarrow \hat{S}$, such that $ S = \hat{S} / \iota$. 
Since the action of $\mathbb{Z}/2\mathbb{Z}$ is free on $\hat{S}$, $\mathbb{Z}/2\mathbb{Z}$  also acts freely on $Z$.

\noindent
Since $Z^3 = 12$, $Z$  is a divisor of type $(1,1,2)$ in $A$.  The covering involution $ \iota \colon \hat{S} \rightarrow \hat{S}$
can be lifted to an involution $g$ of $A$, which we write as an affine transformation  $ g (a) = \alpha a + \beta$. 

\noindent
 We have now Abelian subvarieties $A_1 = \ker (\alpha - \Id)$, $A_2 = \ker (\alpha + \Id)$,
and since the irregularity of $S$ equals $1$, $A_1$ has dimension $1$, and $A_2$ has dimension $2$.

\noindent
We observe preliminarly  that $g$ is fixed point free: since otherwise the fixed point locus would be non empty 
 of  dimension one (as there is exactly one eigenvalue equal to $1$), so it would intersect the ample divisor $Z$, contradicting that $ \iota \colon Z \rightarrow Z$ acts freely.

 \noindent
 Therefore $Y = Z / \iota $ is a divisor in the Bagnera-de Franchis threefold $ X = A / G$, where $G$ is the group of order two generated by $g$.
 
 \noindent
 We can then write the Abelian threefold $A$ as $ (A_1 \times A_2) / T$, and since $\beta_1 \notin T_1$ (cf. Proposition \ref{quotprodtype}) we have only  two possible cases:
 
 \begin{itemize}
 \item[0)] $T = 0$, or
 \item[1)] $T \cong \mathbb{Z}/2\mathbb{Z}$.
 \end{itemize}
 We further observe that, since  the divisor $ Z$ is $g$-invariant, its polarization is $\alpha$ invariant,
in particular its Chern class $c \in \wedge ^2 ( \Hom (\Lambda, \mathbb{Z}))$, where $A = V / \Lambda$.
Since $ T =  \Lambda / ( \Lambda_1 \oplus \Lambda_2)$, $c$ pulls back to 
$$c'   \in \wedge ^2 ( \Hom ( \Lambda_1 \oplus \Lambda_2, \mathbb{Z})) = \wedge ^2 ( \Lambda_1^{\vee})  \oplus \wedge ^2 ( \Lambda_2^{\vee}) \oplus (\Lambda_1^{\vee})\otimes (\Lambda_2^{\vee}) ,$$
and by invariance $c'  = (c'_1 \oplus c'_2 ) \in \wedge ^2 ( \Lambda_1^{\vee})  \oplus \wedge ^2 ( \Lambda_2^{\vee})$. So Case 0) bifurcates in the following cases: 
\begin{itemize}
 \item[0-I)] $c'_1$ is of type $(1)$, $c'_2$ is of type $(1,2)$;
\item[0-II)] $c'_1$ is of type $(2)$, $c'_2$ is of type $(1,1)$.
\end{itemize}
Both cases can be discarded, since they lead to the same contradiction. Setting  $D: = Z$, then $D$ is  the  divisor of zeros on 
$A = A_1 \times A_2$ of a section of a line bundle $L$ which is an  exterior tensor product of $L_1$ and $L_2$.
Since
$$ H^0 (A, L) =  H^0 (A_1, L_1)  \otimes  H^0 (A_2, L_2) ,$$
and $H^0 (A_1, L_1)$ has dimension one in case 0-I), while  $H^0 (A_2, L_2)$ has dimension one in case 0-II), 
we conclude that $D$ is a reducible divisor, a contradiction, since $D$ is smooth and connected.

\noindent
In case 1), we denote $A' : = A_1 \times A_2$, and we let $D$ be the inverse image of $Z$ inside $A'$.
Again $D$ is smooth and connected, since $\pi_1(\hat{S})$ surjects onto $\Lambda$. Now $ D^2 = 24$, so the Pfaffian
of $c'$ equals $4$, and there are a priori several  possibilities:

\begin{itemize}
\item[1-I)] $c'_1$ is of type $(1)$;
\item[1-II)] $c'_2$ is of type $(1,1)$;
\item[1-III)] $c'_1$ is of type $(2)$, $c'_2$ is of type $(1,2)$.
\end{itemize}
The cases 1-I) and 1-II) can be excluded as case 0), since $D$ would then be reducible. 
 
 \noindent
 We are then left only with case 1-III), and we may, without loss of generality, assume that 
 $H^0 (A_1, L_1)=  H^0 (A_1, \mathcal{O}_{A_1} (2 O))$.
Moreover, we have already assumed that we have chosen the origin so that all the sections of
 $H^0 (A_2, L_2)$ are even.
 
 \noindent
We have $ A = A' / T $, and we may write the generator of $T$ as $t_1 \oplus t_2$,
and write $ g (a_1 \oplus a_2 ) = (a_1 + \beta_1) \oplus ( a_2 - \beta_2)$.

\noindent
By the description of Bagnera-de Franchis varieties (cf. Proposition \ref{quotprodtype}) we have that $t_1$ and $\beta_1$ are  a basis of the group
of $2$ torsion points of the elliptic curve $A_1$.

Since all sections of $L_2$ are even, the divisor $D$ is $ G \times T$-invariant if and only if it
is invariant under $T$ and under translation by $\beta$. 

\noindent
This condition however implies that translation of $L_2$ by $\beta_2$  is isomorphic to $L_2$, 
and similarly for $t_2$. 
It follows that $\beta_2, t_2$ form a  basis of  $K_2:= \ker (\phi_{L_2}\colon A_2 \rightarrow \Pic^0 (A_2))$,
where  $\phi(y) = t_yL_2 \otimes L_1^{-1}$.
The isomorphism of $G \times T$ with both $K_1 : = A_1 [2]$ and $K_2$ allows to identify
both $ H^0 (A_1, L_1) $ and $ H^0 (A_2, L_2)$ with the Stone-von Neumann representation $L^2 (T,\mathbb{C})$:
observe in fact that there is only one alternating function $(G \times T)  \rightarrow \mathbb{Z}/2\mathbb{Z}$,  independent of the chosen basis.

\noindent
Therefore, there are exactly $4$ invariant divisors in the linear system $|L|$. 
Explicitly, if  $ H^0 (A_1, L_1) $ has basis $x_0, x_1$ and $ H^0 (A_2, L_2) $ has basis $y_0, y_1$,
then the invariant divisors correspond to the four eigenvectors
$$  x_0 y_0 + x_1  y_1\,, \quad x_0 y_0 - x_1  y_1\,, \quad  x_0 y_1 + x_1  y_0\,, \quad  x_0 y_1 - x_1  y_0\,. $$

To prove irreducibility of the above family of surfaces, it suffices to show that all the four 
invariant divisors occur in the same connected family.

\noindent
To this purpose, we just observe that the monodromy of the family of elliptic curves $E_{\tau} : = \mathbb{C} / ( \mathbb{Z} \oplus \mathbb{Z} \tau)$
on the upper half plane has the effect that a transformation in $\SL ( 2 , \mathbb{Z})$ acts on the subgroup $E_{\tau}  [2]$ of points
of $2$-torsion by its image matrix in $\GL ( 2 , \mathbb{Z}/2\mathbb{Z})$, and in turn the effect on the Stone-von Neumann representation is the one
of twisting it by a character of $E_{\tau}  [2]$.

\noindent
This concludes the proof that the moduli space is irreducible of dimension $4$, since the moduli space of elliptic curves,
respectively  the moduli space of Abelian surfaces with a  polarization of  type $(1,2)$, are irreducible, 
of respective dimensions $1$, $3$.

\medskip
\noindent
The final assertion is a consequence of the fact that $\Alb(S) = A_1 / (T_1 + \langle \langle \beta_1 \rangle \rangle)$,
so that the fibres of the Albanese map are just divisors in $A_2$ of type $(1,2)$. Their self intersection 
equals $4 = 2 (g-1)$, hence $g=3$.

In order to establish that the general curve is non hyperelliptic, it suffices to prove the following lemma.

\begin{lemma}
Let $A_2$ be an Abelian surface, endowed with a divisor $L$ of type $(1,2)$, so that there is an isogeny
of degree two $f \colon A_2 \rightarrow A'$ onto a principally polarised Abelian surface, and $L = f^*(\Theta)$.
Then the only curves $C \in |L|$ which are hyperelliptic are contained in  the pull backs of a translate of
$\Theta$ by a point of order $2$ for a suitable such isogeny $f' \colon A_2 \rightarrow A''$.
In particular, the general curve $C \in |L|$ is not hyperelliptic.
\end{lemma}
\begin{proof}
Note that $A'$ is the quotient of $A$ by an involution, given by translation with a two torsion element $t \in A[2]$.
Let $C \in |L|$, and consider $D : = f_* (C) \in | 2 \Theta|$. There are two cases:
\begin{itemize}
\item[I)]  $ C + t = C$;  
\item[II)] $ C + t \neq  C$.  
\end{itemize}
In case I)  $D = 2 B$, where
$B$ has genus $2$, so that $ C = f^* (B)$, hence, since $ 2B \equiv 2 \Theta$, $B$ is a translate of $\Theta$ by a point of order $2$. There are exactly two such curves, and for them $ C \rightarrow B$ is \'etale.

\noindent
In case II) the map $C \rightarrow D$ is birational, $f^* (D) = C \cup (C + t)$. Now, $C+t$ is also  linearly equivalent to $ L$,
hence $C$ and $C+t$ intersect in the $4$ base points of the pencil $|L|$.  Hence $D$ has two double  points and geometric genus 
equal to $3$. These double points are the intersection points of $\Theta$ and a translate of $\Theta$ by a point of order $2$,
and are points of $2$-torsion.

The sections of $H^0(\mathcal{O}_{A'} (2 \Theta))$ are all even 
and $ | 2 \Theta|$ is the pull-back of the space of hyperplane sections of the Kummer surface
$\mathcal{K} \subset \mathbb{P}^3$, the quotient $\mathcal{K} = A' / \{\pm 1\}$.

\noindent
Therefore the image $E'$ of each such curve $D$ lies in the pencil of planes through $2$ nodes of $\mathcal{K}$.

$E'$ is a plane quartic, hence $E'$ has geometric genus $1$, and we conclude that $C$ admits an involution
$\sigma$ with quotient an elliptic curve $E$ (normalization of $E'$), and the double cover is branched in $4$ points.

Assume that $C$ is hyperelliptic, and denote by $h$ the hyperelliptic involution, which lies in the centre of $\Aut(C)$.
Hence we have $(\mathbb{Z}/2\mathbb{Z})^2$ acting on $C$, with quotient $\mathbb{P}^1$. We easily see that there are exactly six
branch points, two being the branch points of $ C/h \rightarrow \mathbb{P}^1$, four being the branch points of $E \rightarrow \mathbb{P}^1$.
It follows that there is an \'etale quotient $ C \rightarrow B$ , where $B$ is the genus $2$ curve, double cover of $\mathbb{P}^1$
branched on the six points.
 
Now, the inclusion $ C \subset A_2$ and the degree $2$ map $ C \rightarrow B$ induce a degree two isogeny $ A_2 \rightarrow J(B)$,
and $C$ is the pull back of the Theta divisor of $J(B)$, thus it cannot be a general curve.
\end{proof}
This ends the proof of Theorem \ref{BCF2}.
\end{proof}
We shall give the surfaces of Theorem \ref{BCF2} a name.
\begin{definition}
A  minimal surface $S$ of general type with invariants $K_S^2 = 6$, $p_g(S) =q(S) = 1$ such that 
\begin{itemize}
\item there exists an unramified double cover
$ \hat{S} \rightarrow S$ with $ q (\hat{S}) = 3$, and such that 
\item the Albanese morphism $ \hat{\alpha} \colon  \hat{S}  \rightarrow A = \Alb(\hat{S})$ is birational onto its image $Z$,
a divisor in $A$ with  $ Z^3 = 12$,
\end{itemize} is called a {\it Sicilian  surface with $q(S)=p_g(S)=1$}.
\end{definition}

\begin{remark}
We have seen that the canonical model of a Sicilian surface $S$ is an ample divisor in a Bagnera-de Franchis threefold $X =A/G$, where $G =\langle g \rangle \cong \mathbb{Z}/ 2 \mathbb{Z}$. Hence the fundamental group of $S$ is isomorphic to the  fundamental group $\Gamma$ of $X$. Moreover, $\Gamma$ fits into 
the exact sequence 
$$  1 \longrightarrow \Lambda \longrightarrow  \Gamma \longrightarrow G = \mathbb{Z}/2\mathbb{Z} \longrightarrow 1\,,$$ 
and is generated by the union of the set $\{ g , t\}$ with the set of translations by the elements of a basis $\lambda_1, \lambda_2, \lambda_3, \lambda_4$ of $\Lambda_2$, where
$$ g ( v_1 \oplus  v_2 ) = ( v_1 + \frac{\tau}{2} )  \oplus ( - v_2 + \frac{\lambda_2}{2} )$$ 
$$ t ( v_1 \oplus  v_2 ) = ( v_1 + \frac 12 ) \oplus (v_2 + \frac{\lambda_4}{2} ).$$

$\Gamma$ is therefore a semidirect product of $\mathbb{Z}^5 = \Lambda_2 \oplus \mathbb{Z} t$ with the infinite cyclic group generated by $g$:
conjugation by $g$ acts as $-1$ on $\Lambda_2$, and it sends $t \mapsto t - \lambda_4$ (hence $ 2 t - \lambda_4$ is an eigenvector for the eigenvalue $1$).
\end{remark}

We shall now give a topological characterization of Sicilian surfaces with $q=p_g=1$, following the lines
of \cite{BC12}.

\noindent
Observe in this respect that $X$ is a $K(\Gamma, 1)$-space, so that its cohomology and homology are just 
group cohomology, respectively homology, of the group $\Gamma$. 

\begin{corollary}\label{he}
A Sicilian surface $S$ with $q(S)=p_g(S)=1$ is characterized by the following properties:

\begin{enumerate}
\item
$K_S^2 = 6$,
\item
$ \chi(S) = 1$,
\item
$\pi_1(S) \cong \Gamma$, where $\Gamma$ is as above,
\item
the classifying map $f \colon S \rightarrow X$, where $X$ is the Bagnera-de Franchis threefold which is a classifying space for $\Gamma$,
has the property that $ f_* [S]  = :B$ satisfies $ B^3=6$.
\end{enumerate}

\noindent
In particular, any surface homotopically equivalent to a Sicilian surface is a Sicilian surface, and we get a connected component of the moduli space of surfaces
of general type which is stable under the action of the absolute Galois group $Gal (\bar{\QQ}, \QQ)$.

\end{corollary}

\begin{proof}
Since $\pi_1(S) \cong \Gamma$, first of all $ q(S) =1$, hence also $p_g(S) = 1$.
By the same token
there is a double \'etale cover $\hat{S} \rightarrow S$ such that $ q(\hat{S} ) = 3$,
and the Albanese image of  $\hat{S} $, counted with multiplicity, is the inverse image of $B$,
therefore $Z^3 = 12$. From this, it follows that $\hat{S} \rightarrow Z$ is birational, since the class of $Z$ is
indivisible.

\noindent
We may now apply the previous Theorem \ref{BCF2} in order to obtain the classification.

\noindent
Observe finally that the condition $(\hat{\alpha}_* \hat{S})^3 = 12$ is not only a topological condition,
it is also invariant under  $Gal (\bar{\QQ}, \QQ)$.
\end{proof}

\section{Proof of the main theorems}

We conclude in this last short section the proofs of Main Theorem 1 and   Main Theorem 2.

\begin{proof}[Proof of Main Theorem 1]
Statements 1), 2) and 3) summarize the contents of Proposition \ref{onefam} and Theorem \ref{fundgroup}.

\noindent
4) We observe preliminarly that our fundamental groups are virtually Abelian of rank 6 (i.e., they have a normal subgroup of finite index $\cong \ZZ^6$). By the results  of \cite{4names}, the fundamental group of (the minimal resolution of) a product-quotient surface has a finite index normal subgroup which is the product of at most two fundamental groups of curves. Therefore if it is virtually Abelian it has rank 2 or 4.

This argument excludes rightaway that our fundamental groups may be isomorphic 
to the  fundamental groups of some product-quotient surfaces. 

The only remaining case for $p_g=0$ is the Kulikov surface, whose first homology group has $3$-torsion. \footnote{Disclaimer: the fundamental group of the Inoue surface with $p_g=0$, $K^2=6$ has not yet been calculated and we do not claim it is different from  ours.}

The known surfaces with $p_g=q=1$ and $K^2=6$ are either product-quotient surfaces (cf. \cite{pol09}) or mixed quasi-\'etale surfaces, which are constructed in \cite{FP14}. Comparing Table 2 from loc. cit with our Table \ref{q1}, we see that they have different homology groups from ours.

\noindent
5) is proved in Theorem \ref{moduli}.
\end{proof}

\begin{proof}[Proof of Main Theorem 2]
The assertions 1) and 2) are contained in Theorem \ref{BCF2}. 

\noindent
4) is contained in Corollary \ref{he}.

\noindent
3) Observe that in cases $\mathcal S_{11}$ and $\mathcal S_{12}$ of Table \ref{q1} there is a subgroup $H \cong (\ZZ/2 \ZZ)^2$ acting  by translations on $E_1\times E_2 \times E_3$. Denote by $\hat S$ the quotient of the Burniat hypersurface by $H$. Then $\hat S$ is an \'etale double cover of the GBT $S$, which satisfies the defining property of Sicilian surfaces.

\noindent
There remains to show that the other GBT surfaces (with $p_g=q=1$) are not Sicilian surfaces. This is now obvious since they have  fundamental groups non-isomorphic to $\pi_1(S_{11})$, where $S_{11}$ belongs to the family $\mathcal S_{11}$ .
\end{proof}

\bibliographystyle{alpha} 


\

\noindent{\bf Authors' Adresses:}\\
{\it Ingrid Bauer, Fabrizio Catanese, Davide Frapporti}\\
Lehrstuhl Mathematik VIII\\
Mathematisches Institut der Universit\"at Bayreuth, NW II\\
Universit\"atsstr. 30; D-95447 Bayreuth, Germany.

\medskip 
\noindent{\bf Email Adresses:}\\
Ingrid.Bauer@uni-bayreuth.de\\
Fabrizio.Catanese@uni-bayreuth.de\\
Davide.Frapporti@uni-bayreuth.de

\newpage

\appendix

\section{Tables}

\begin{table}[!h]
\renewcommand\arraystretch{.9}
\begin{tabular}{c|ccc|ccc|ccc|c}
 & $\epsilon_0$ & $\eta_1$& $\epsilon_1$  & $\eta_0$ & $\eta_1$& $\epsilon_2$ & $\zeta_0$ & $\eta_1$& $\epsilon_3$& $H_1$\\
 \hline

\multirow{3}{*}{$\mathcal S_1$}&1&0&0&1&0&0&1&0&0 &\multirow{3}{*}{	$(\mathbb Z/ 2\mathbb Z)^2\times (\mathbb Z/ 4\mathbb Z)^2$}	  \\
&0&1&0&1&1&0&1&1&0&\\
&0&0&0&0&0&1&1&0&1&\\
\hline
 \multirow{3}{*}{$\mathcal S_2$}&1&0&0&0&0&1&1&0&1&\multirow{3}{*}{$(\mathbb Z/ 2\mathbb Z)^6$}	\\
&0&0&1&0&0&0&1&0&1&\\
&0&0&0&1&0&1&0&0&1&\\

	\hline
\multirow{3}{*}{$\mathcal S_3$}&1&0&0&0&0&1&1&0&1&\multirow{3}{*}{$ (\mathbb Z/ 4\mathbb Z)^3$	}\\
&0&1&0&0&1&0&1&1&0&\\
&0&0&1&1&0&1&1&0&0&\\
\hline
\multirow{3}{*}{$\mathcal S_4$}&1&0&1&0&0&1&1&0&0&\multirow{3}{*}{$(\mathbb Z/ 2\mathbb Z)^2\times (\mathbb Z/ 4\mathbb Z)^2$}	\\
&0&1&0&0&1&0&1&1&0&\\
&0&0&0&1&0&1&1&0&1&\\
\hline

&&&&&&&&&&\\	
\end{tabular}\caption{$q=0$}\label{q0}
\end{table}

\begin{table}[!h]\renewcommand\arraystretch{.9}
\begin{tabular}{c|ccc|ccc|ccc|c|c}
& $\epsilon_0$ & $\eta_1$& $\epsilon_1$  & $\eta_0$ & $\eta_1$& $\epsilon_2$ & $\zeta_0$ & $\eta_1$& $\epsilon_3$& $H_1$& $\pi_1$\\
\hline
\multirow{3}{*}{$\mathcal S_5$}&1&0&1&0&0&0&1&0&1&\multirow{3}{*}{$(\mathbb Z/ 2\mathbb Z)^3\times \mathbb Z^2$}&\\
&0&1&0&0&1&0&1&1&0&&\\
&0&0&0&0&0&1&1&0&1&&\\
		\hline
\multirow{3}{*}{$\mathcal S_{6}$}&0&1&0&1&1&0&1&1&0&\multirow{3}{*}{$(\mathbb Z/ 2\mathbb Z)^2\times \mathbb Z^2$}&\\
&0&0&1&1&0&0&1&0&1&&\\
&0&0&0&0&0&1&1&0&1	&&\\
	\hline
	\multirow{3}{*}{$\mathcal S_{7}$}&1&0&0&0&0&1&1&0&1&\multirow{3}{*}{$(\mathbb Z/ 4\mathbb Z)\times \mathbb Z^2$}&\\
&0&1&0&0&1&0&1&1&0&&\\
&0&0&1&1&0&1&0&0&0	&&\\

	\hline
\multirow{3}{*}{$\mathcal S_{8}$}&1&0&0&0&0&1&1&0&1&\multirow{3}{*}{$(\mathbb Z/ 2\mathbb Z)^2\times \mathbb Z^2$}&\\
&0&1&0&1&1&0&1&1&0&&\\
&0&0&1&1&0&0&1&0&1&&	\\
\hline
\multirow{3}{*}{$\mathcal S_{9}$}&1&0&0&1&0&1&1&0&1&\multirow{3}{*}{$(\mathbb Z/ 2\mathbb Z\times \mathbb Z/ 4\mathbb Z)\times \mathbb Z^2$}&\\
&0&1&0&0&1&0&1&1&0&&\\
&0&0&1&1&0&1&0&0&0&&\\
\hline
\multirow{3}{*}{$\mathcal S_{10}$}&1&0&1&1&0&0&1&0&1&\multirow{3}{*}{$(\mathbb Z/ 2\mathbb Z)^2\times \mathbb Z^2$}&\\
&0&1&0&1&1&0&1&1&0&&\\
&0&0&0&0&0&1&1&0&1&&\\
	\hline
	\multirow{3}{*}{$\mathcal S_{11}$}&1&0&0&1&0&1&0&0&1&\multirow{3}{*}{$(\mathbb Z/ 2\mathbb Z)^3\times \mathbb Z^2$}&\\
&0&1&0&1&1&0&0&1&0&&\\
&0&0&1&1&0&1&1&0&0&&\\	
\hline
\multirow{3}{*}{$\mathcal S_{12}$}&1&0&1&0&0&0&1&0&1&\multirow{3}{*}{$(\mathbb Z/ 2\mathbb Z)^3\times \mathbb Z^2$}&\multirow{3}{*}{$\cong\pi_1(S_{11})$}\\
&0&1&0&0&1&0&1&1&0&&\\
&0&0&0&1&0&1&1&0&1&&\\

	\hline
&&&&&&&&&&&\\	

\end{tabular}\caption{$q=1$}\label{q1}
\end{table}

\begin{table}[!h]\renewcommand\arraystretch{.9}
\begin{tabular}{c|ccc|ccc|ccc|c|c}
& $\epsilon_0$ & $\eta_1$& $\epsilon_1$  & $\eta_0$ & $\eta_1$& $\epsilon_2$ & $\zeta_0$ & $\eta_1$& $\epsilon_3$& $H_1$& $\pi_1$\\
\hline

\multirow{3}{*}{$\mathcal S_{13}$}&1&0&0&1&0&1&1&0&1&\multirow{3}{*}{$ \mathbb Z^4$}&\\
&0&1&0&1&1&0&1&1&0&&\\
&0&0&1&1&0&1&0&0&0&&\\

\hline

\multirow{3}{*}{$\mathcal S_{14}$}&1&0&1&0&0&0&0&0&1&\multirow{3}{*}{$(\mathbb Z/ 2\mathbb Z)\times \mathbb Z^4$}&\\
&0&1&1&0&1&1&0&1&0&&\\
&0&0&0&1&0&1&0&0&1&&\\
\hline

\multirow{3}{*}{$\mathcal S_{15}$}&1&0&1&0&0&0&1&0&1& \multirow{3}{*}{$(\mathbb Z/ 2\mathbb Z)\times \mathbb Z^4$}&\multirow{3}{*}{$\cong\pi_1(S_{14})$}\\
&0&1&1&0&1&1&0&1&0&&\\
&0&0&0&1&0&1&1&0&1&&\\

	\hline
&&&&&&&&&&&\\	

\end{tabular}\caption{$q=2$}\label{q2}
\end{table}

\begin{table}[!h]\renewcommand\arraystretch{.9}
\begin{tabular}{c|ccc|ccc|ccc|c}
& $\epsilon_0$ & $\eta_1$& $\epsilon_1$  & $\eta_0$ & $\eta_1$& $\epsilon_2$ & $\zeta_0$ & $\eta_1$& $\epsilon_3$& $\pi_1$\\
\hline
\multirow{3}{*}{$\mathcal S_{16}$}&1&0&1&0&0&0&1&0&1& \multirow{3}{*}{$\mathbb Z^6$}\\
&0&1&1&0&1&1&1&1&0&\\
&0&0&0&1&0&1&1&0&1&\\

	\hline
&&&&&&&&&&\\	

\end{tabular}\caption{$q=3$}\label{q3}
\end{table}

\end{document}